\documentclass[11pt]{amsart}
\usepackage[hmargin=2.5cm,vmargin=2.5cm]{geometry}
\usepackage{amsfonts}
\usepackage{amsthm}
\usepackage[english]{algorithm2e}
\usepackage{amsmath}
\usepackage{amssymb}
\usepackage{hyperref}
\usepackage[T1]{fontenc}
\newtheorem{Theorem}{Theorem}[section]
\newtheorem{theorem}[Theorem]{Theorem}
\newtheorem{definition}[Theorem]{Definition}
\newtheorem{remark}[Theorem]{Remark}
\newtheorem{example}[Theorem]{Example}
\newtheorem{lemma}[Theorem]{Lemma}
\newtheorem{proposition}[Theorem]{Proposition}
\newtheorem{corollary}[Theorem]{Corollary}

\title{Representations and relative Rota-Baxter operators of
Hom-Leibniz-Poisson algebras}
\author[Sylvain Attan ]
       { Sylvain Attan}
\begin{document}
\maketitle
\begin{abstract}
Representations and relative Rota-Baxter operators with respect to representations of Hom-Leibniz Poisson algebras are introduced and studied. Some characterizations of these operators are obtained. The notion of matched pair and Nijenhuis operators of Hom-Leibniz Poisson algebras are given and various relevant constructions of these Hom-algebras are deduced.
\end{abstract}

\section{Introduction}
Leibniz algebras are a non-commutative version of Lie algebras whose brackets
satisfy the Leibniz identities rather than Jacobi identities \cite{jlLlP}. As Poisson algebras have simultaneousely  associative and Lie algebras structures, Leibniz Poisson algebras \cite{mctd} have simultaneousely associative and Leibniz algebras structures. Hence, they can be viewed as  a non-commutative version of Poisson algebras. Recall that a (non-commutative) Leibniz-Poisson algebra is
an associative (non-commutative) algebra $
(A,\cdot)$ equipped with a binary bracket operation $[,]: A^{\times 2}\rightarrow A$ such that $(A,[,])$ is a Leibniz algebra and the following compatibility condition holds
$$[x\cdot y,z]=x\cdot[y,z]+[x,y]\cdot z \mbox{ for all $x,y,z\in A.$}$$
The general theory of Hom-algebras takes its roots in the introduction of
Hom-Lie algebras by D. Larsson, S. D. Silvestrov and J. T. Hartwig \cite{dlsds1}, \cite{dlsds2}, \cite{HAR1}. Starting from the well-known relation between associative algebras and Lie algebras, the notion of Hom-associative algebras was introduced \cite{MAK3} and it was proved that the commutator Hom-algebra of any Hom-associative algebra is a Hom-Lie algebra. Since then, other types of Hom-algebras have emerged; in particular, Hom-Leibniz algebras \cite{MAK3} were introduced as a non-commutative version of Hom-Lie algebras as well as Hom-Poisson algebras \cite{dy1} which have simultaneousely Hom-associative algebra and Hom-Lie algebra structures and satisfying a certain compatibility condition.
Hom-Leibniz algebras have been widely studied from the point of view of representation and cohomology theory \cite{ycys}, deformation theory \cite{sahhbk}\cite{gmrs} in recent years. The same is true for Hom-associative algebras since, the definition of
Hochschild-type cohomology  and the study of the one parameter formal deformation theory for these type of
Hom-algebras are given \cite{fazeam}, \cite{amss}.

The notion of Rota-Baxter operators on associative algebras was introduced \cite{gb} and those on Leibniz algebras are considered in \cite{ysrt}. It is found that, those operators have many applications in Connes-Kreimers algebraic approach to the renormalization in perturbative quantum field theory \cite{acdk}. 
A generalization of Rota-Baxter operators called relative Rota-Baxter operators (or $\mathcal{O}$-operators) has been introduced for left Hom-Leibniz algebras \cite{sgsw} where the graded Lie algebra that
characterizes those operators as Maurer-Cartan elements is constructed whereas in    
 \cite{tcsmam}, the cohomology theory of 
$\mathcal{O}$-operators on Hom-
associative algebras are found. In this paper, from the representations of Hom-associative and Hom-Leiniz algebras, we will establish those of  Hom-Leibniz-Poisson algebras. We will also introduce  relative Rota-Baxter operators on Hom-Leibniz Poisson algebras after having considered them on Hom-associative and right Hom-Leibniz algebras. Most of the results for relative Rota-Baxter operators on right Hom-Leibniz algebras are established since they differ of course from those for left Hom-Leibniz algebras \cite{sgsw}.

The paper is organized as follows. 
Section 2 is devoted to reminders of fundamental concepts. Some results on representations and relative Rota-Baxter operators are proved for Hom-associative algebras. Concerning section 3,  similar results to those of the previous section are established for (right) Hom-Leibniz algebras. Although some of these results are proved in the case of left Hom-Leibniz algebras \cite{sgsw}, we take them again in our case for consistency of the rest. Finally, the last section contains the main results of this work. It is a logical continuation of the results of the previous sections. Here, we introduce and study representations of Hom-Leibniz Poisson algebras. The notions of matched pairs and relative Rota-Baxter operators of such Hom-algebras have also been discussed and interesting results have been obtained. 

Throughout this paper, all vector spaces and algebras are meant over a ground field $\mathbb{K}$ of characteristic 0.
\section{Preliminaries}
This section is devoted to some definitions which are a very useful  for next sections. Some elemnetary results are also proven.
\begin{definition}
A Hom-module is a pair $(A,\alpha_M)$ consisting of a $\mathbb{K}$-module $A$ and 
a linear self-map $\alpha_A: A\longrightarrow A.$ A morphism 
$f: (A,\alpha_A)\longrightarrow (B,\alpha_B)$ of Hom-modules is a linear map 
 $f: A\longrightarrow B$ such that $f\circ\alpha_A=\alpha_B\circ f.$
\end{definition}
\begin{definition} A Hom-algebra is a triple $(A,\mu,\alpha)$ in which $(A,\alpha)$ is a Hom-module, $\mu: A^{\otimes 2}\longrightarrow A$ is a linear map.
The Hom-algebra $(A,\mu,\alpha)$ is said to be  multiplicative if $\alpha\circ\mu=\mu\circ\alpha^{\otimes 2}.$  A morphism 
$f: (A,\mu_A,\alpha_A)\longrightarrow (B,\mu_B,\alpha_B)$ of Hom-algebras is a morphism of the underlying Hom-modules such that $f\circ\mu_A=\mu_B\circ f^{\otimes 2}.$
\end{definition}
\begin{definition}
 Let $(A,\mu,\alpha)$ be a Hom-algebra and 
 $\lambda\in\mathbb{K}.$  Let $R$ be a linear map satisfying 
 \begin{eqnarray}
  \mu(R(x),R(y))=R(\mu(R(x),y)+\mu(x,R(y))+\lambda\mu(x,y)),\ \forall x,y \in A \label{Rota-Baxt}
 \end{eqnarray}
Then, $R$ is called a Rota-Baxter operator of weight $\lambda$ and $(A,\mu,\alpha, R)$ is called a Rota-Baxter Hom-algebra of weight 
$\lambda.$
\end{definition}
In the sequel, to unify our terminologies by a Rota-Baxter operator (resp. a Rota-Baxter Hom-algebra), we mean a Rota-Baxter operator (resp. a Rota-Baxter Hom-algebra) of weight $\lambda=0.$

\begin{definition}\cite{MAK3}
 A Hom-associative algebra is a multiplicative Hom-algebra $(A,\mu,\alpha)$
satisfying the Hom-associativity condition i.e.,
\begin{eqnarray}
 as_{\alpha}(x,y,z):=\mu(\mu(x,y),\alpha(z))-\mu(\alpha(x),\mu(y,z))=0\label{HAs} \mbox{ for all $x,y,z\in A$}
\end{eqnarray}
 \end{definition}
%\begin{example}
%\end{example}
Other Hom-algebras as Hom-associative algebras, which occur regularly in this paper are Hom-Leibniz algebras.
\begin{definition}\cite{MAK3}
 A (right) Hom-Leibniz algebra is a multiplicative Hom-algebra $(A,[,],\alpha)$ satisfying the Hom-Leibniz identity
 \begin{eqnarray}
  [[x,y],\alpha(z)]=[\alpha(x),[y,z]]+
  [[x,z],\alpha(y)] \mbox{ for all $x,y,z\in A$} \label{leib}
 \end{eqnarray}
\end{definition}
%\begin{example}
%\end{example}
Recall that a morphism $f: (A,\mu_A,\alpha_A)\rightarrow (B,\mu_B,\alpha_B)$ of  Hom-associative (resp. Hom-Leibniz) algebras is a morphism of underlying Hom-algebras.\\

Hom-Leibniz Poisson algebras are fundamental Hom-algebras of this paper. They are a generalization of Hom-associative and  Hom-Leibniz algebras. These Hom-algebras can also be viewed as a non-anticommutative version of Hom-Poisson algebras.
\begin{definition}
 A Hom-Leibniz Poisson algebra is a triple
 $(A,\cdot,[,],\alpha)$ consisting of a linear space $A,$ two bilinear maps $\cdot, [,]: A^{\times 2} \rightarrow A$ and a linear map $\alpha: A\rightarrow A$ satisfying the following axioms
 \begin{enumerate}
  \item $(A,\cdot,\alpha)$ is a Hom-associative algebra,
\item $(A, [,],\alpha)$ is a Hom-Leibniz algebra,
\item the following  condition holds
\begin{eqnarray}
 [x\cdot y, \alpha(z)]=\alpha(x)\cdot [y,z]+[x,z]\cdot\alpha(y)\mbox{ for all $x,y,z\in A$} \label{cHLP}
\end{eqnarray}
 \end{enumerate}
 \end{definition}
 A morphism $f: (A,\cdot_{A},[,]_A,\alpha_A)\rightarrow (B,\cdot_{B},[,]_B,\alpha_B)$ of Hom-Leibniz Poisson algebras is a morphism of underlying Hom-associative and Hom-Leibniz algebras. 
 \begin{example}
 \begin{enumerate}
  \item  Let $\mathcal{A}:=(A,\cdot,[,])$ be a Leibniz-Poisson algebra and $\alpha$ be a self-morphism of $\mathcal{A}.$ Then, 
  $\mathcal{A}_{\alpha}:=(A,\cdot_{\alpha}:=\cdot\circ\alpha^{\otimes 2}, [,]_{\alpha}:=[,]\circ\alpha^{\otimes 2},\alpha)$ is a Hom-Leiniz-Poisson algebra.
  \item Any Hom-Poisson algebra is a Hom-Leibniz-Poisson algebra.
 \end{enumerate}
 \end{example}
  \begin{remark}
   \begin{enumerate}
   \item  A Hom-associative algebra $(A,\cdot,\alpha)$ (resp. Hom-Leibniz algebra $(A,[,],\alpha$) is called Rota-Baxter Hom-associative (resp. Rota-Baxter Hom-Leibniz) algebra if the underlying Hom-algebra $(A,\cdot,\alpha)$ (resp. $(A,[,],\alpha)$) is a Rota-baxter Hom-algebra.
   \item Similarly, a Hom-Leibniz Poisson algebra $(A,\cdot,[,],\alpha)$ is called a Rota-Baxter Hom-Leibniz Poisson algebra if $((A,\cdot,\alpha)$ is a Rota-Baxter Hom-associative algebra and $(A,[,],\alpha)$ is a Rota-Baxter Hom-Leibniz algebra.
   \end{enumerate}
  \end{remark}
\begin{definition} 
 \begin{enumerate}
  \item A two-sided Hom-ideal of a Hom-associative (resp. Hom-Leibniz) algebra $(A,\mu,\alpha)$ is a subspace $I$ of $A$ satisfying $\alpha(I)\subseteq I,$ $\mu(I,A)\subseteq I$ and $\mu(A,I)\subseteq I.$
  \item A two-sided Hom-ideal of a Hom-Leibniz Poisson algebra $(A, \cdot,[,],\alpha)$ is a two-sided Hom-ideal of the Hom-associative algebra $(A, \cdot,\alpha)$ and a two-sided Hom-ideal of the Hom-Leibniz algebra $(A,[,],\alpha).$
 \end{enumerate}
\end{definition}
Let recall the definition of representations of Hom-associative algebras and some results about these notions.
\begin{definition}
 A representation of a Hom-associative algebra $(A, \cdot, \alpha)$  is a quadruple $(V,\lambda^l,\lambda^r,\phi)$ where $V$ is a vector space, $\phi\in gl(V)$ and $\lambda^l, \lambda^r: A\rightarrow gl(V)$ are three linear maps
such that the following equalities hold for all $x, y \in A:$
\begin{eqnarray}
\phi\lambda^l(x)=\lambda^l(\alpha(x))\phi; \
 \phi\lambda^r(x)=\lambda^r(\alpha(x))\phi\label{rAs1}\\
 \lambda^l(x\cdot y)\phi=\lambda^l(\alpha(x))\lambda^l(y)\label{rAs2}\\
 \lambda^r(x\cdot y)\phi=\lambda^r(\alpha(y))\lambda^r(x)\label{rAs3}\\
 \lambda^l(\alpha(x))\lambda^r(y)=\lambda^r(\alpha(y))\lambda^l(x)\label{rAs4}
\end{eqnarray}
\end{definition}
To give examples of representations of Hom-associative algebras, let prove the following:
\begin{proposition}\label{PEx1}
 Let $\mathcal{A}_1:=(A_1,\mu_1,\alpha_1)$ and $\mathcal{A}_2:=(A_2,\mu_2,\alpha_2)$ be two Hom-associative algebras and $f:\mathcal{A}_1\rightarrow \mathcal{A}_2$ a morphism of Hom-associative algebras. Then $A_2^f:=(A_2, \lambda^l,\lambda^r,\alpha_2)$ is a representation of $\mathcal{A}_1$ where
 $\lambda^l(a)b:=\mu_2(f(a),b)$ and $\lambda^r(a,b):=\mu_2(b,f(a))$ for all $(a,b)\in A_1\times A_2.$
\end{proposition}
\begin{proof}
 Let $(x,y)\in A_1^{\times 2}$ and $z\in A_2.$ First, using $f$ is a morphism we prove (\ref{rAs1}) as follows:
 \begin{eqnarray}
  \alpha_2(\lambda^l(x)z)=\alpha_2\mu_2(f(x),z)=\mu_2(\alpha_2f(x),\alpha_2(z))=\mu_2(f\alpha_1(x),\alpha_2(z))=\lambda^l(\alpha_1(x)\alpha_2(z).\nonumber
 \end{eqnarray}
Hence, $\alpha_2\lambda^l(x)=\lambda^l(\alpha_1(x)\alpha_2$ and similarly $\alpha_2\lambda^r(x)=\lambda^r(\alpha_1(x)\alpha_2.$   Next, as $f$ is a morphism, using (\ref{HAs}) in $\mathcal{A}_2,$ we proceed for (\ref{rAs2}) as follows:
\begin{eqnarray}
&& \lambda^l(\mu_1(x,y))\alpha_2(z)=\mu_2(f(\mu_1(x,y)),\alpha_2(z))=\mu_2(\mu_2(f(x),f(y)),\alpha_2(z))\nonumber\\
 &&=\mu_2(\alpha_2f(x),\mu_2(f(y),z))=\mu_2(f\alpha_1(x),\mu_2(f(y),z))=\lambda^l(\alpha_1(x))\lambda^l(y)z\nonumber
\end{eqnarray}
Therefore, $\lambda^l(\mu_1(x,y))\alpha_2=\lambda^l(\alpha_1(x))\lambda^l(y).$ In a similar way, (\ref{rAs3}) holds. Finally, thanks to conditions $f\alpha_1=\alpha_2f$ and (\ref{HAs}) in $\mathcal{A}_2,$ we compute
\begin{eqnarray}
&& \lambda^l(\alpha_1(x))\lambda^r(y)z=
 \mu_2(f\alpha_1(x),\mu_2(z,f(y))=
 \mu_2(\alpha_2f(x),\mu_2(z,f(y))
 \nonumber\\
 &&=
 \mu_2(\mu_2(f(x),z),\alpha_2f(y))=
 \mu_2(\mu_2(f(x),z),f\alpha_1(y))=
 \lambda^r(\alpha_1(y))\lambda^r(x)z\nonumber
\end{eqnarray}
and therefore, (\ref{rAs4}) is obtained.
\end{proof}
Now, using Proposition \ref{PEx1}, we obtain the following example as applications.
\begin{example}
\begin{enumerate}
 \item Let $(A, \cdot, \alpha)$ be a Hom-associative algebra.
 Define a left multiplication 
 $ l: A\rightarrow gl(A)$ and a
  right multiplication $r: A\rightarrow gl(A)$ by $l(x)y:=x\cdot y$ and $ r(x)y:=y\cdot x$  for all $x, y \in A.$ Then $(A , l,r, \alpha)$ is a
representation of $(A, \cdot, \alpha),$ which is called a regular representation
 \item Let $(A,\mu,\alpha)$ be a Hom-associative algebra and $(B,\alpha)$ be a two-sided Hom-ideal of $(A,\mu,\alpha).$ Then $(B,\alpha)$  inherits a structure of representation of $(A,\mu,\alpha)$ where 
 $\rho^l(a)b:=\mu(a,b);\ \rho^r(b,a):=\mu(b,a)$ for all $(a,b)\in A\times B.$
\end{enumerate}
\end{example}
\begin{proposition}\label{sRHas}
 Let $\mathcal{V}:=(V, \lambda^l,\lambda^r,\phi)$ be a representation of a Hom-associative algebra $\mathcal{A}:=(A,\cdot,\alpha)$ and $\beta$ be a self-morphism of $\mathcal{A}$. Then 
 $\mathcal{V}_{\beta}=:(V,\lambda^l_{\beta}:=\lambda^l\beta,\lambda^r_{\beta}:=\lambda^r\beta,\phi)$ is a representation of $\mathcal{A}.$ 
\end{proposition}
\begin{proof} Let $x,y\in A.$
 First, by (\ref{rAs1}) in $\mathcal{V}$ and the condition $\alpha\beta=\beta\alpha$ we have 
 \begin{eqnarray}
  \phi\lambda^l_{\beta}(x)=\phi\lambda^l(\beta(x))=\lambda^l(\alpha\beta(x))\phi=\lambda^l(\beta\alpha(x))\phi=\lambda^l_{\beta}(\alpha(x))\phi.\nonumber
 \end{eqnarray}
Similarly, we prove $\phi\lambda^r_{\beta}(x)=\lambda^r_{\beta}(\alpha(x))\phi$ and hence, we obtain (\ref{rAs1}) for $\mathcal{V}_{\beta}.$ Next, by (\ref{rAs2}) in $\mathcal{V}$ and the fact that $\beta$ is a morphism, we get:
\begin{eqnarray}
 &&\lambda^l_{\beta}(x\dot y)\phi=\lambda^l(\beta(x\cdot y))\phi=\lambda^l(\beta(x)\cdot\beta(y))\phi=\lambda^l(\alpha\beta(x))\lambda^l(\beta(y))\nonumber\\
 &&=\lambda^l(\beta\alpha(x))\lambda^l(\beta(y))=\lambda^l_{\beta}(x)\lambda^l_{\beta}(y)\nonumber
\end{eqnarray}
i.e., (\ref{rAs2}) holds and similarly, (\ref{rAs3}) holds for $\mathcal{V}_{\beta}.$  Finally, using $\alpha\beta=\beta\alpha$  and (\ref{rAs4}) for $\mathcal{V},$ we compute
\begin{eqnarray}
 \lambda^l_{\beta}(\alpha(x))\lambda^r_{\beta}(y)=\lambda^l(\beta\alpha(x))\lambda^l(\beta(y))=\lambda^l(\alpha\beta(x))\lambda^l(\beta(y))=\lambda^r(\alpha\beta(y))\lambda^l(\beta(x))=
 \lambda^r_{\beta}(\alpha(y))\lambda^l_{\beta}(x)\nonumber
\end{eqnarray}
\end{proof}
Let us recall the following necessary results  for the last section of this paper.
\begin{proposition}\label{spHAs}\cite{mnhhs}
 Let $(A,\cdot,\alpha)$ be a Hom-associative  algebra. Then $(V, \lambda^l,\lambda^r, \phi)$ is a representation of $(A, \cdot,\alpha)$ if and only if the direct sum of vector spaces, $A\oplus V,$  turns into a 
Hom-associative 
algebra with the multiplication and the linear map defined by
\begin{eqnarray}
 (x+u)\ast(y+v):=x\cdot y+(\lambda^l(x)v+\lambda^r(y)u)\label{SdpHAs1}\\
 (\alpha\oplus\phi)(x+u):=\alpha(x)+\phi(u)\label{sdpHAs2}
\end{eqnarray}
This Hom-associative algebra is called the semi-direct product of $A$ with $V.$
\end{proposition}
\begin{definition}\cite{mnhhs}
 Let $\mathcal{A}_1:=(A_1,\cdot,\alpha_1)$ and 
 $\mathcal{A}_2:=(A_2,\bullet,\alpha_2)$ two Hom-associative algebras. Let $\lambda_1^l, \lambda_1^r: A_1\rightarrow gl(A_2)$ and 
 $\lambda_2^l, \lambda_2^r: A_2\rightarrow gl(A_1)$ be linear maps such that $(A_2,\lambda_1^l, \lambda_1^r, \alpha_2)$ is a representation of $\mathcal{A}_1,$ 
 $(A_1,\lambda_2^l, \lambda_2^r, \alpha_2)$ is a representation of $\mathcal{A}_2$ and the following conditions hold
 \begin{eqnarray}
  &&\lambda_1^l(\alpha_1(x)(u\bullet v)=\lambda_1^l(\lambda_2^r(u)x)\alpha_2(v)+(\lambda_1^l(x)u)\bullet\alpha_2(v),\nonumber\\
&&\lambda_1^r(\alpha(x))(u\bullet v)=
\lambda_1^r(\lambda_2^l(v)x)\alpha_2(u)+
\alpha_2(u)\bullet(\lambda_1^r(x)v),\nonumber\\
&&\lambda^l(\alpha_2(u))(x\cdot y)=\lambda_2^l(\lambda_1^r(x)u)\alpha_1(y)+
(\lambda_2^l(\alpha_2(u)x)\cdot\alpha_1(y),\nonumber\\
&&\lambda_2^r(\alpha_2(u))(x\cdot y)=\lambda_2^r(\lambda_1^l(y)u)\alpha_1(x)+
\alpha_1(x)\cdot(\lambda_2^r(u)y),\nonumber\\
&&\lambda_1^l(\lambda_2^l(u)x)\alpha_2(v)+
(\lambda_1^r(x)u)\bullet\alpha_2(v)-
\lambda_1^r(\lambda_2^r(v)x)\alpha_2(u)-
\alpha_2(u)\bullet(\lambda_1^l(x)v)=0,\nonumber\\
&&\lambda_2^l(\lambda_1^l(x)u)\alpha_1(y)+
(\lambda_2^r(u)x)\cdot\alpha_1(y)-
\lambda_2^r(\lambda_1^r(y)u)\alpha_1(x)-
\alpha_1(x)\cdot(\lambda_2^l(u)y)=0,\nonumber
 \end{eqnarray}
Then $((A_1, \lambda_2^l, \lambda_2^r,\alpha_1), (A_2,\lambda_1^l, \lambda_1^r,\alpha_2))$ is called a matched pair of  Hom-associative algebras.
\end{definition}
\begin{proposition}\cite{mnhhs}
 Let $((A_1, \lambda_2^l, \lambda_2^r,\alpha_1), (A_2,\lambda_1^l, \lambda_1^r,\alpha_2))$ be a matched pair of Hom-associative algebras. Then, there is a Hom-associative algebra $(A_1\oplus A_2, \ast, \alpha_1\oplus\alpha_2)$ defined by
\begin{eqnarray}
 (x+b)\ast(y+b)&:=&(x\cdot y+\lambda_2^l(a)y+\lambda_2^r(b)x)+(a\bullet b+\lambda_1^l(x)b+\lambda_1^r(y)a) \label{opHAs1}\\
 (\alpha_1\oplus\alpha_2)(x+a)&:=&\alpha_1(x)+\alpha_2(a)\nonumber
\end{eqnarray}
\end{proposition}
Recall the following notion known as $\mathcal{O}$-operator for Hom-associative algebras in the literature.
\begin{definition}
 Let $(V,\lambda^l,\lambda^r,\phi )$ be a representation of a Hom-associative algebra $(A,\cdot,\alpha).$ A linear
operator $T : V\rightarrow A$ is called a relative Rota-Baxter operator on $(A,\cdot,\alpha)$ with respect to $(V,\lambda^l,\lambda^r,\phi )$ if $T$ satisfies
\begin{eqnarray}
 && T\phi=\alpha T \label{rbHAs1}\\
  && (Tu)\cdot(Tv)=T(\lambda^l(Tu)v+\lambda^r(Tv)u) \mbox{ for all $u,v\in V$} \label{rbHAs2}
\end{eqnarray}
\end{definition}
Observe  that Rota-Baxter operators on Hom-associative algebras are relative Rota-Baxter operators with respect to the regular representation.
\begin{example}\label{exHAsr}
 Consider the $2$-dimensional  Hom-associative algebra 
$(A,\cdot,\alpha)$  where the non-zero products with respect to a basis $(e_1,e_2)$ are given by: $e_1\cdot e_2=e_2\cdot e_1:=-e_1,\ e_2\cdot e_2:=e_1+e_2;$ and $\alpha(e_1):=-e_1,\ \alpha(e_2):=e_1+e_2.$
\\ Then a linear map $T: A\rightarrow A$ defined by  $T(e_1):=a_{11}e_1+a_{21}e_2;\ 
T(e_2):=a_{12}e_1+a_{22}e_2$ is a relative Rota-Baxter on $(A,\cdot,\alpha)$ with respect to the regular representation  if and only if $T\alpha=\alpha T$ and
\begin{eqnarray}
(Te_i)\cdot(Te_j)=T((Te_i)\cdot e_j+e_i\cdot (Te_j)) 
\mbox{ for all $i,j\in \{1,2\}.$}\label{exHAsreq}
\end{eqnarray}
 The condition $\alpha T=T\alpha$ is equivalent to 
\begin{eqnarray}
 a_{21}=0;\  a_{11}+2a_{12}-a_{22}=0.\nonumber
\end{eqnarray}
For $i=j=1,$ the condition (\ref{exHAsreq}) is satisfied trivially.
Similarly, we obtain for \\ $(i,j)\in\{(1,2),(2,1)\}$ 
\begin{eqnarray}
 -a_{11}a_{12}=-a_{11}a_{12}-a_{11}^2;\nonumber 
 \end{eqnarray} 
 and for $i=j=2,$
 \begin{eqnarray}
 -2a_{12}a_{22}+a_{22}^2=-2a_{12}a_{11}+2a_{22}a_{11}+2a_{22}a_{12};\  a_{22}^2=4a_{22}^2.\nonumber
\end{eqnarray}
Summarize the above discussions, we observe that the zero-map  is the only  relative Rota-Baxter operator on $(A,\cdot,\alpha)$  with respect to the regular representation.\\

Until the end of this section, we state and prove some results affirmed in \cite{tcsmam}. These results will be used in the last section on this paper.
\end{example}
\begin{lemma}\label{lHAsRB}
 Let $T$ be a relative Rota-Baxter operator on a Hom-associative algebra $(A,\cdot,\alpha)$ with respect to a representation  $(V,\lambda^l,\lambda^r,\phi ).$ If define a map "$\diamond$" on $V$ by 
 \begin{eqnarray}
  u\diamond v:=\lambda^l(Tu)v+\lambda^r(Tv)u \mbox{ for all $(u, v)\in V^{\times 2}$} \label{opRBHas}
 \end{eqnarray}
then, $(V,\diamond, \phi)$ is a Hom-associative algebra.
\end{lemma}
\begin{proof}
 First, note that the multiplicativity of 
 $\diamond$ with respect to $\phi$ follows from conditions (\ref{rAs1}) and (\ref{rbHAs1}). Next, pick $u,v,w\in V$ and observe from (\ref{opRBHas}) and (\ref{rbHAs2}) that $T(u\diamond v)=Tu\cdot Tv.$ Therefore, by a straightforward computations:
 \begin{eqnarray}
  &&(u\diamond v)\diamond\phi(v)=\lambda^l(Tu\cdot Tv)\phi(w)+\lambda^r(T\phi(w))\lambda^l(Tu)v+\lambda^r(T\phi(w))\lambda^r(Tv)u.\nonumber
 \end{eqnarray}
Similarly, after rearranging terms
\begin{eqnarray}
 &&\phi(u)\diamond (v\diamond w)=\lambda^l(T\phi(u))\lambda^l(Tv)w+
 \lambda^l(T\phi(u))\lambda^r(Tw)v+
 \lambda^r(Tv\cdot Tw)\phi(u)\nonumber
\end{eqnarray}
Hence, the Hom-associativity in $(V,\diamond,\phi)$ follows by (\ref{rbHAs1}) and  $(\ref{rAs2})-(\ref{rAs4}).$
\end{proof}
\begin{corollary}\label{cmorpHas}
 Let $T$ be a relative Rota-Baxter operator on a Hom-associative algebra $(A,\cdot,\alpha)$ with respect to a representation  $(V,\lambda^l,\lambda^r,\phi ).$ Then, $T$ is a morphism from the Hom-associative algebra 
 $(V,\diamond,\phi)$ to the initial Hom-associative algebra $(A,\cdot,\alpha).$
\end{corollary}
\begin{theorem}\label{tHasRB}
 Let $T$ be a relative Rota-Baxter operator on a Hom-associative algebra $(A,\cdot,\alpha)$ with respect to a representation  $(V,\lambda^l,\lambda^r,\phi ).$  
 Then, $(A,\overline{\lambda^l}, \overline{\lambda^r},\alpha)$ is a representation of the Hom-associative algebra
 $(V,\diamond,\phi)$ where
 \begin{eqnarray}
 && \overline{\lambda^l}(u)x:=(Tu)\cdot x-T\lambda^r(x)u   \mbox{ for all $(x,u)\in A\times V$}\label{rrba1}\\
 && \overline{\lambda^r}(u)x:=x\cdot(Tu)-T\lambda^l(x)u  \mbox{ for all $(x,u)\in A\times V$}\label{rrba2}
 \end{eqnarray}
 \end{theorem}
 \begin{proof}
  First,  (\ref{rAs1}) in $(A,\overline{\lambda^l}, \overline{\lambda^r},\alpha)$ follows from the one in $(V,\lambda^l,\lambda^r,\phi ),$  the multiplicativity of $\cdot$ with respect to $\alpha$ and (\ref{rbHAs1}). Next, for all 
  $x\in A$ and $(u,v)\in V^{\times 2},$ if one observes from Corollary \ref{cmorpHas} that $T(u\diamond v)=Tu\cdot Tv,$ we compute:
  \begin{eqnarray}
   &&\overline{\lambda^l}(u\diamond v)\alpha(x)=(Tu\cdot Tv)\cdot\alpha(x)-T\lambda^r(\alpha(x))\lambda^l(Tu)v-T\lambda^r()\alpha(x))\lambda^r(Tv)u\nonumber\\
   &&=(T\phi(u))\cdot(Tv\cdot x)-T\lambda^l(T\phi(u))\lambda^r(x)v-T\lambda^r()\alpha(x))\lambda^r(Tv)u \mbox{ ( by (\ref{rbHAs1}), (\ref{HAs}) and (\ref{rAs4}) ).}\nonumber
  \end{eqnarray}
Similarly, by straightforward computations:
\begin{eqnarray}
 &&\overline{\lambda^l}(\phi(u))\overline{\lambda^l}(v)x=(T\phi(u))\cdot(Tv\cdot x)-T\phi(u)\cdot T\lambda^r(x)v-T\lambda^r(Tv\cdot x)\phi(u)\nonumber\\&&+T\lambda^r(T\lambda^r(x)v)\phi(u)=
 (T\phi(u))\cdot(Tv\cdot x)-
 T\lambda^l(T\phi(u))\lambda^r(x)v-T\lambda^r(T\lambda^r(x)v)\phi(u)\nonumber\\
 &&-T\lambda^r(Tv\cdot x)\phi(u)+T\lambda^r(T\lambda^r(x)v)\phi(u) \mbox{ ( by (\ref{rbHAs2}) )}\nonumber\\
 &&=(T\phi(u))\cdot(Tv\cdot x)-
 T\lambda^l(T\phi(u))\lambda^r(x)v-T\lambda^r(Tv\cdot x)\phi(u).\nonumber
\end{eqnarray}
Hence, we obtain (\ref{rAs2}) in $(A,\overline{\lambda^l}, \overline{\lambda^r},\alpha)$ by (\ref{rAs3}) in 
 $(V,\lambda^l,\lambda^r,\phi ).$ Next to prove (\ref{rAs3}) in  $(A,\overline{\lambda^l}, \overline{\lambda^r},\alpha)$, let proceed  as the previous case, i.e.,
 \begin{eqnarray}
  &&\overline{\lambda^r}(u\diamond v)\alpha(x)=\alpha(x)\cdot(Tu\cdot Tv)-
  T\lambda^l(\alpha(x))\lambda^l(Tu)v-T\lambda^l(\alpha(x))\lambda^r(Tv)u,\nonumber\\
  &&\overline{\lambda^r}(\phi(v))\overline{\lambda^r}(u)x=(x\cdot Tu)\cdot T\phi(v)-(T\lambda^l(x)u)\cdot T\phi(v)-
  T\lambda^l(x\cdot Tu)\phi(v)+T\lambda^l(T\lambda^l(x)u)\phi(v)\nonumber\\
  &&(x\cdot Tu)\cdot T\phi(v)-T\lambda^r(T\phi (v))\lambda^l(x)u -
  T\lambda^l(x\cdot Tu)\phi(v) \mbox{ ( by (\ref{rbHAs2}) ).}\nonumber
 \end{eqnarray}
Hence, we obtain (\ref{rAs3}) by (\ref{rbHAs1}) and (\ref{HAs}), (\ref{rAs2}), (\ref{rAs4}) in $(V,\lambda^l,\lambda^r,\phi ).$ Finally, we compute:
 \begin{eqnarray}
  &&\overline{\lambda^l}(\phi(u))\overline{\lambda^r}(v)x=T\phi(u)\cdot(x\cdot Tv)-T\phi(u)\cdot T\lambda^l(x)v-T\lambda^r(x\cdot Tv)\phi(u)+T\lambda^r(T\lambda^l(x)v)\phi(u)\nonumber\\
  &&=T\phi(u)\cdot(x\cdot Tv)-T\lambda^l(T\phi(u))\lambda^l(u)v-T\lambda^r(T\lambda^l(x)v)\phi(u)-T\lambda^r(x\cdot Tv)\phi(u)\nonumber\\&&+T\lambda^r(T\lambda^l(x)v)\phi(u) \mbox{ ( by (\ref{rbHAs2}) )}\nonumber\\
  &&=T\phi(u)\cdot(x\cdot Tv)-T\lambda^l(T\phi(u))\lambda^l(u)v-T\lambda^r(x\cdot Tv)\phi(u).\nonumber
 \end{eqnarray}
Similarly, after a straightforward computation, we get also by (\ref{rbHAs2}):
\begin{eqnarray}
 &&\overline{\lambda^r}(\phi(v))\overline{\lambda^l}(u)x= (Tu\cdot x)\cdot T\phi(v)-T\lambda^r(T\phi(v))\lambda^r(x)u-T\lambda^l(Tu\cdot x)\phi(v).\nonumber
\end{eqnarray}
Therefore, we obtain the desired identity by (\ref{rbHAs1}), (\ref{HAs}), (\ref{rAs2}) and (\ref{rAs3}).
 \end{proof}
\section{Matched pair and relative Rota-Baxter operator on Hom-Leibniz algebras}
This section is primarily dedicated to the study of relative Rota-Baxter operators for (right) Hom-Leibniz algebras.  Although some of the results obtained here are treated for left Hom-Leibniz algebras \cite{sgsw}, we will also treat them in the right Hom-Leibniz algebras. Indeed, these were needed for the understanding of the last section which is full of the fundamental results of this paper.
\begin{definition}\cite{ycys}
 A representation of a Hom-Leibniz algebra $(A, [,], \alpha)$  is a quadruple $(V,\phi,\rho^l,\rho^r)$ where $V$ is a vector space, $\phi\in gl(V)$ and $\rho^l, \rho^r: A\rightarrow gl(V)$ are three linear maps
such that the following equalities hold for all $x, y \in A:$
\begin{eqnarray}
 \rho^l(\alpha(x))\phi=\phi\rho^l(x);\ 
 \rho^l(\alpha(x))\phi=\phi\rho^l(x) \label{rHL1}\\
 \rho^l([x,y])\phi=\rho^l(\alpha(x))\rho^l(y)+\rho^r(\alpha(y))\rho^l(x)\label{rHL2}\\
 \rho^r(\alpha(y))\rho^l(x)=
 \rho^l(\alpha(x))\rho^r(y)
 +\rho^l([x,y])\phi\label{rHL3} \\
 \rho^r(\alpha(y))\rho^r(x)=\rho^r([x,y])\phi+\rho^r(\alpha(x))\rho^r(y) \label{rHL4}
\end{eqnarray}
\end{definition}
Observe that  (\ref{rHL4}) gives rise to the useful equation:
\begin{eqnarray}
 \rho^r([x,y])\phi+\rho^r([y,x])\phi=0 \mbox{ for all $(x,y)\in A^{\times 2}$}\label{rHL4r}
\end{eqnarray}
To give examples of representations of Hom-Leibniz algebras, let prove:
\begin{proposition}\label{PEx2}
 Let $\mathcal{A}_1:=(A_1,[,]_1,\alpha_1)$ and $\mathcal{A}_2:=(A_2,[,]_2,\alpha_2)$ be two Hom-Leibniz algebras and $f:\mathcal{A}_1\rightarrow \mathcal{A}_2$ be a morphism of Hom-Leibniz algebras. Then $(A_2, \rho^l,\rho^r,\alpha_2)$ is a representation of $\mathcal{A}_1$ where
 $\rho^l(a)b:=[f(a),b]_2$ and $\rho^r(a)b:=[b,f(a)]_2$ for all $(a,b)\in A\times B.$
\end{proposition}
\begin{proof}
 As (\ref{rAs1}) in Proposition \ref{PEx1}, (\ref{rHL1}) holds. Next, let $(x,y)\in A_1^{\times 2}$ and $z\in A_2.$  Then,  using $f$ is a morphism and (\ref{leib}), we compute 
 \begin{eqnarray}
  &&\rho^l([x,y]_1)\alpha_2(z)=
  [f([x,y]_1),\alpha_2(z)]_2=
  [[f(x),f(y)]_2,\alpha_2(z)]_2=
  [\alpha_2f(x),[f(y),z]_2]_2\nonumber\\
  &&+[[f(x),z]_2,\alpha_2f(y)]_2=
  [f\alpha_1(x),[f(y),z]_2]_2+[[f(x),z]_2,f\alpha_1(y)]_2\nonumber\\
  &&=\rho^l(\alpha_1(x))\rho^l(y)z+\rho^r(\alpha_1(y))\rho^l(x)z \mbox{ and } \nonumber\\
  &&\rho^r(\alpha_1(y))\rho^l(x)z=[[f(x),z]_2,f\alpha_1(y)]_2=
  [[f(x),z]_2,\alpha_2f(y)]_2=
  [\alpha_2f(x),[z,f(y)]_2]_2\nonumber\\
  &&+[[f(x),f(y)]_2,\alpha_2(z)]_2=
  [f\alpha_1(x),[z,f(y)]_2]_2+[f([x,y]_1),\alpha_2(z)]_2\nonumber\\
 && =\rho^l(\alpha_1(x))\rho^r(y)z+\rho^l([x,y]_1)\alpha_2(z).\nonumber  
 \end{eqnarray}
Therefore, we get (\ref{rHL2}) and (\ref{rHL3}). In a similar way, we prove (\ref{rHL4}).
\end{proof}

As Hom-associative case, using Proposition \ref{PEx2}, we get:
\begin{example}
\begin{enumerate}
 \item Let $(A, [,], \alpha)$ be a Hom-Leiniz algebra.
 Define a left multiplication 
 $ l: A\rightarrow gl(A)$ and a
  right multiplication $r: A\rightarrow gl(A)$ by $l(x)y:=[x,y]$ and $ r(x)y:=[y,x]$  for all $x, y \in A.$ Then $(A, l,r,\alpha)$ is a
representation of $(A, [,], \alpha),$ which is called a regular representation.
 \item Let $(A,[,],\alpha)$ be a Hom-Leibniz  algebra and $(B,\alpha)$ be a two-sided Hom-ideal of $(A,[,],\alpha).$ Then $(B,\alpha)$  inherits a structure of representation of $(A,[,],\alpha)$ where 
 $\rho^l(a)b:=[a,b];\ \rho^r(b,a):=[b,a]$ for all $(a,b)\in A\times B.$
\end{enumerate}
\end{example}
\begin{proposition}\label{sRHLeib}
 Let $\mathcal{V}:=(V, \rho^l,\rho^r,\phi)$ be a representation of a Hom-Leibniz algebra $\mathcal{A}:=(A,[,],\alpha)$ and $\beta$ be a self-morphism of $\mathcal{A}:=(A,[,],\alpha)$. Then 
 $\mathcal{V}_{\beta}:=(V,\rho^l_{\beta}:=\rho^l\beta,\rho^r_{\beta}:=\rho^r\beta,\phi)$ is a representation of $\mathcal{A}.$ 
\end{proposition}
\begin{proof} Let $x,y\in A.$
 First, by (\ref{rHL1}) in $\mathcal{V}$ and the condition $\alpha\beta=\beta\alpha$ we have 
 \begin{eqnarray}
  \phi\rho^l_{\beta}(x)=\phi\rho^l(\beta(x))=\rho^l(\alpha\beta(x))\phi=\rho^l(\beta\alpha(x))\phi=\rho^l_{\beta}(\alpha(x))\phi.\nonumber
 \end{eqnarray}
Similarly, we prove $\phi\rho^r_{\beta}(x)=\rho^r_{\beta}(\alpha(x))\phi$ and hence, we obtain (\ref{rHL1}) for $\mathcal{V}_{\beta}.$ Next, by (\ref{rHL2}) in $\mathcal{V}$ and the fact that $\beta$ is a morphism, we get:
\begin{eqnarray}
 &&\rho^l_{\beta}([x,y])\phi=\rho^l(\beta([x,y]))\phi=\rho^l([\beta(x),\beta(y)])\phi=\rho^l(\alpha\beta(x))\rho^l(\beta(y))\nonumber\\
 &&+\rho^r(\alpha\beta(y))\rho^l(\beta(x))
 =\rho^l(\beta\alpha(x))\rho^l(\beta(y))
 +\rho^r(\beta\alpha(y))\rho^l(\beta(x))\nonumber\\
 &&=\rho^l_{\beta}(\alpha(x))\rho^l_{\beta}(y)+\rho^r_{\beta}(\alpha(y))\rho^l_{\beta}(x)\nonumber
\end{eqnarray}
i.e., (\ref{rHL2}) holds  for $\mathcal{V}_{\beta}.$ 
 Finally, using $\alpha\beta=\beta\alpha$  and (\ref{rHL3}) for $\mathcal{V},$ we compute
\begin{eqnarray}
 &&\rho^r_{\beta}(\alpha(y))\rho^l_{\beta}(x)=\rho^r(\beta\alpha(y))\rho^l(\beta(x))
 =\rho^r(\alpha\beta(y))\rho^l(\beta(x))=
 \rho^l(\alpha\beta(x))\rho^r(\beta(y))\phi\nonumber\\
 &&+\rho^l([\beta(x),\beta(y)])=
 \rho^l(\beta\alpha(x))\rho^r(\beta(y))+\rho^l(\beta([x,y]))=
 \rho^l_{\beta}(\alpha(x))\rho^r_{\beta}(y)+\rho^l_{\beta}([x,y]).\nonumber
\end{eqnarray}
Therefore, we obtain (\ref{rHL3})  and similarly (\ref{rHL4}) for $\mathcal{V}_{\beta}.$ 
\end{proof}
\begin{proposition}\label{spHLa}
 Let $(A,[,],\alpha)$ be a Hom-Leibniz algebra and let $(V,\phi)$ be a Hom-module. Let $ \rho^l,\rho^r : A\rightarrow gl(V)$ be two linear maps. The quadruple $(V,\rho^l,\rho^r,\phi )$ is a representation of $(A,[,],\alpha)$  if
and only if the direct sum of vector spaces $A\oplus V$   turns into a 
Hom-Leibniz algebra by defining a multiplication and a linear map  by
\begin{eqnarray}
  \{(x+u),(y+v)\}:=[x,y]+(\rho^l(x)v+\rho^r(y)u)\label{sdpHL1}\\
 (\alpha\oplus\phi)(x+u):=\alpha(x)+\phi(u)\label{sdpHL2}
\end{eqnarray}
\end{proposition}
\begin{proof}
 Firs observe  that the multiplicativity of
 $\alpha\oplus\phi$ with respect to $\{,\}$ is equivalent to (\ref{rHL1}).
 Next, pick $x,y,z\in A$ and $u,v,w\in V.$ Then, by a straightforward computation, we obtain
 \begin{eqnarray}
  &&\{\{x+u,y+v\},(\alpha\oplus\phi)(z+w)\}=\underbrace{[[x,y],\alpha(z)]}_{(i)}+
  \rho^l([x,y])\phi(w)\nonumber\\
  &&+\rho^r(\alpha(z))\rho^l(x)v+\rho^r(\alpha(z))\rho^r(y)u,\nonumber\\
 && \{(\alpha\oplus\phi)(x+u),\{y+v,z+w\}\}
 =[\alpha(x),[y,z]]+\rho^l(\alpha(x))\rho^l(y)w+\rho^l(\alpha(x))\rho^r(z)v\nonumber\\
 &&+\rho^r([y,z])\phi(u),\nonumber\\
 &&\{\{x+u,y+v\},(\alpha\oplus\phi)(z+w)\}=
 [[x,z],\alpha(y)+\rho^l([x,z])\phi(v)+\rho^r(\alpha(y))\rho^l(x)w\nonumber\\
 &&+\rho^r(\alpha(y))\rho^r(z)u.\nonumber
 \end{eqnarray}
Hence, by (\ref{leib}) for (i), we get:
\begin{eqnarray}
 &&\{\{x+u,y+v\},(\alpha\oplus\phi)(z+w)\}-
 \{(\alpha\oplus\phi)(x+u),\{y+v,z+w\}\}
 \nonumber\\
 &&-\{\{x+u,y+v\},(\alpha\oplus\phi)(z+w)\}
 =\Big(\rho^l([x,y])\phi(w)-\rho^l(\alpha(x))\rho^l(y)w-\rho^r(\alpha(y))\rho^l(x)w\Big)
 \nonumber\\
 &&\Big(\rho^r(\alpha(z))\rho^l(x)v-\rho^l(\alpha(x))\rho^r(z)v-\rho^l([x,z])\phi(v)\Big)+\Big(\rho^r(\alpha(z))\rho^r(y)u-\rho^r([y,z])\phi(u)\nonumber\\
 &&-\rho^r(\alpha(y))\rho^r(z)u\Big).\nonumber
\end{eqnarray}
Therefore, (\ref{leib}) holds if and only if (\ref{rHL2}), (\ref{rHL3}) and (\ref{rHL4}) hold.
\end{proof}
 Let give the following:
\begin{definition}
 Let $\mathcal{A}_1:=(A_1, [,]_1,\alpha_1)$ and 
 $\mathcal{A}_2:=(A_2, [,]_2,\alpha_2)$ be two Hom-Leibniz algebras. Let $\rho_1^l,\rho_1^r: A_1\rightarrow gl(A_2)$ and 
 $\rho_2^l,\rho_2^r: A_2\rightarrow gl(A_1)$ be linear maps such that $(A_2,\rho_1^l,\rho_1^r, \alpha_2)$ is a representation of $\mathcal{A}_1,$ 
 $(A_1,\rho_2^l,\rho_2^r, \alpha_1)$ is a representation of $\mathcal{A}_2$ and the following conditions hold
 \begin{eqnarray}
  &&\rho_1^r(\alpha_1(x))[u,v]_2-[\alpha_2(u),\rho_1^r(x)v]_2-[\rho_1^r(x)u,\alpha_2(v)]_2
  -\rho_1^r(\rho_2^l(v)x)\alpha_2(u)\nonumber\\
  &&-\rho_1^l(\rho_2^l(u)x)\alpha_2(v)=0\label{mpHL1}\\
 && \rho_1^l(\alpha_1(x))[u,v]_2-[\rho_1^l(x)u,\alpha_2(v)]_2+[\rho_1^l(x)v,\alpha_2(u)]_2
  -\rho_1^l(\rho_2^r(u)x)\alpha_2(v)\nonumber\\
  &&+
  \rho_1^l(\rho_2^r(v)x)\alpha_2(u)=0\label{mpHL2}\\
  &&\rho_1^r(\alpha_1(x))[u,v]_2-[\rho_1^r(x)u,\alpha_2(v)]_2+[\alpha_2(u),\rho_1^l(x)v]_2-\rho_1^l(\rho_2^l(u)x)\alpha_2(v)\nonumber\\
  &&+
  \rho_1^r(\rho_2^r(v)x)\alpha_2(u)=0
  \label{mpHL3}\\
  &&\rho_2^r(\alpha_2(u))[x,y]_1-[\alpha_1(x),\rho_2^r(u)y]_1-[\rho_2^r(u)x,\alpha_1(y)]_1
  -\rho_2^r(\rho_1^l(y)u)\alpha_2(x)\nonumber\\
  &&-\rho_2^l(\rho_1^l(x)u)\alpha_1(y)=0\label{mpHL4}\\
  && \rho_2^l(\alpha_2(u))[x,y]_1-[\rho_2^l(u)x,\alpha_1(y)]_2+[\rho_2^l(u)y,\alpha_1(x)]_1
  -\rho_2^l(\rho_1^r(x)u)\alpha_1(y)\nonumber\\
  &&+
  \rho_2^l(\rho_1^r(y)u)\alpha_1(x)=0\label{mpHL5}\\
  &&\rho_2^r(\alpha_2(u))[x,y]_1-[\rho_2^r(u)x,\alpha_1(y)]_1+[\alpha_1(x),\rho_2^l(u)y]_1-\rho_2^l(\rho_1^l(x)u)\alpha_1(y)\nonumber\\
  &&+\rho_2^r(\rho_1^r(y)u)\alpha_1(x)=0\label{mpHL6}
 \end{eqnarray}
Then $((A_1,\rho_2^l,\rho_2^r, \alpha_1),(A_2,\rho_1^l,\rho_1^r, \alpha_2))$ is called a matched pair of Hom-Leibniz algebras.
\end{definition}
\begin{remark}
 \begin{enumerate}
  Observe that (\ref{mpHL1}) and (\ref{mpHL3}) give rise to
  \begin{eqnarray}
   &&[\alpha_2(u),\rho_1^l(x)v]_2+\rho_1^r(\rho_2^l(v)x)\alpha_2(u)+
  [\alpha_2(u),\rho_1^r(x)v]_2+\rho_1^r(\rho_2^r(v)x)\alpha_2(u)=0\nonumber
  \end{eqnarray}
  and (\ref{mpHL4}) and (\ref{mpHL6}) give rise to
  \begin{eqnarray}
  &&[\alpha_1(x),\rho_2^l(u)y]_1+\rho_2^r(\rho_1^l(y)u)\alpha_1(x)
  +[\alpha_1(x),\rho_2^r(u)y]_1+\rho_2^r(\rho_1^r(y)u)\alpha_1(x)=0.\nonumber
  \end{eqnarray}
 \end{enumerate}

\end{remark}

\begin{proposition}
 Let $((A_1,\rho_2^l,\rho_2^r, \alpha_1),(A_2,\rho_1^l,\rho_1^r, \alpha_2))$ be a matched pair of Hom-Leibniz algebras. Then, there is a Hom-Leibniz algebra $A_{\Join}:=(A_1\oplus A_2,
\{,\},\alpha_1\oplus\alpha_2)$ defined by
\begin{eqnarray}
 \{(x+u),(y+v)\}&:=&([x,y]_1+\rho_2^l(u)y+\rho_2^r(v)x)+([u,v]+\rho_1^l(x)v+\rho_1^r(y)u) \label{opHLeib1}\\
 (\alpha_1\oplus\alpha_2)(x+u)&:=&\alpha_1(x)+\alpha_2(u)\nonumber
\end{eqnarray}
\end{proposition}
\begin{proof}
 The multiplicativity of $A_{\Join}$ follows from the one of $\alpha_1$ with respect to $[,]_1,$ $\alpha_2$ with respect to $[,]_2$ and the condition (\ref{rHL1}). Next, pick 
 $(x,y,z)\in A_1^{\times 3}$ and 
 $(u,v,w)\in A_2^{\times 3}.$ Then by direct computations
 \begin{eqnarray}
 && \{\{x+u,y+v\},(\alpha_1\oplus\alpha_2)(z+w\}\nonumber\\
 &&=\underbrace{[[x,y]_1,\alpha_1(z)]_1}_{(i_1)}+[\rho_2^l(u)y,\alpha_1(z)]_1+[\rho_2^r(v)x,\alpha_1(z)]_1\nonumber\\
 &&+\underbrace{\rho_2^l([u,v]_2)\alpha_1(z)}_{(i_2)}+\rho_2^l(\rho_1^l(x)v)\alpha_1(z)
 +\rho_2^l(\rho_1^r(y)u)\alpha_1(z)\nonumber\\
 &&+\rho_2^r(\alpha_2(w)[x,y]_1+\underbrace{\rho_2^r(\alpha_2(w))\rho_2^l(u)y}_{(i_3)}
 +\underbrace{\rho_2^r(\alpha_2(w))\rho_2^l(v)x}_{(i_4)}\nonumber\\
 &&+\underbrace{[[u,v]_2,\alpha_2(w)]_2}_{(i_1)}+[\rho_1^l(x)v,\alpha_2(w)]_2+[\rho_1^r(y)u,\alpha_2(w)]_2\nonumber\\
 &&+\underbrace{\rho_1^l([x,y]_1)\alpha_2(w)}_{(i_2)}+\rho_1^l(\rho_2^l(u)y)\alpha_2(w)
 +\rho_1^l(\rho_2^r(v)x)\alpha_2(w)\nonumber\\
 &&+\rho_1^r(\alpha_1(z)[u,v]_2+\underbrace{\rho_1^r(\alpha_1(z))\rho_1^l(x)v}_{(i_3)}
 +\underbrace{\rho_1^r(\alpha_1(z))\rho_1^l(y)u}_{(i_4)},\nonumber
 \end{eqnarray}
 \begin{eqnarray}
 &&\{(\alpha_1\oplus\alpha_2)(x+u),\{y+v,z+w\}\}\nonumber\\
 &&=[\alpha_1(x),[y,z]_1]_1+[\alpha_1(x),\rho_2^l(v)z]_1+[\alpha_1(x),\rho_2^r(w)y]_1\nonumber\\
 &&+\rho_2^l(\alpha_2(u))[y,z]_1+\rho_2^l(\alpha_2(u))\rho_2^l(v)z+
 \rho_2^l(\alpha_2(u))\rho_2^r(w)y\nonumber\\
 &&+\rho_2^r([v,w]_2)\alpha_1(x)+\rho_2^r(\rho_1^l(y)w)\alpha_1(x)
 +\rho_2^r(\rho_1^r(z)v)\alpha_1(x)\nonumber\\
 &&+[\alpha_2(u),[v,w]_2]_2+[\alpha_2(u),\rho_1^l(y)w]_2+[\alpha_2(u),\rho_1^r(z)v]_2\nonumber\\
 &&+\rho_1^l(\alpha_1(x))[v,w]_2+\rho_1^l(\alpha_1(x))\rho_1^l(y)w+
 \rho_1^l(\alpha_1(x))\rho_1^r(z)v\nonumber\\
 &&+\rho_1^r([y,z]_1)\alpha_2(u)+\rho_1^r(\rho_2^l(v)z)\alpha_2(u)
 +\rho_1^r(\rho_2^r(w)y)\alpha_2(u),\nonumber\\
 &&\{\{x+u,y+v\},(\alpha_1\oplus\alpha_2)(z+w)\}\nonumber\\
 &&=[[x,z]_1,\alpha_1(y)]_1+[\rho_2^l(u)z,\alpha_1(y)]_1+[\rho_2^r(w)x,\alpha_1(y)]_1\nonumber\\
 &&+\rho_2^l([u,w]_2)\alpha_1(y)+\rho_2^l(\rho_1^l(x)w)\alpha_1(y)+\rho_2^l(\rho_1^r(z)u)\alpha_1(y)\nonumber\\
 &&+\rho_2^r(\alpha_2(v))[y,z]_1+\rho_2^r(\alpha_2(v))\rho_2^l(u)z
 +\rho_2^r(\alpha_2(v))\rho_2^r(w)x\nonumber\\
 &&+[[u,w]_2,\alpha_2(v)]_2+[\rho_1^l(x)w,\alpha_2(v)]_2+[\rho_1^r(z)u,\alpha_2(v)]_2\nonumber\\
 &&+\rho_1^l([x,z]_1)\alpha_2(v)+\rho_1^l(\rho_2^l(u)z)\alpha_2(v)+\rho_1^l(\rho_2^r(w)x)\alpha_2(v)\nonumber\\
 &&+\rho_1^r(\alpha_1(y))[v,w]_2+\rho_1^r(\alpha_1(y))\rho_1^l(x)w
 +\rho_1^r(\alpha_1(y))\rho_1^r(z)u.\nonumber
 \end{eqnarray}
Hence, by $(\ref{leib}), (\ref{rHL2}),(\ref{rHL3}), (\ref{rHL4})$ for $(i_1), (i_2), (i_3), (i_4)$ respectively, we get:
\begin{eqnarray}
 &&\{\{x+u,y+v\},(\alpha_1\oplus\alpha_2)(z+w\}-\{(\alpha_1\oplus\alpha_2)(x+u),\{y+v,z+w\}\}\nonumber\\
 &&-\{\{x+u,y+v\},(\alpha_1\oplus\alpha_2)(z+w)\}\nonumber\\
 &&=\Big([\rho_2^l(u)y,\alpha_1(z)]_1+\rho_2^l(\rho_1^r(y)u)\alpha_1(z)-\rho_2^l(\alpha_2(u))[y,z]_1-\rho_2^l(\rho_1^r(z)u)\alpha_1(y)\nonumber\\
 &&-[\rho_2^l(u)z,\alpha_1(y)]_1\Big)+
 \Big([\rho_2^r(v)x,\alpha_1(z)]_1+\rho_2^l(\rho_1^l(x)v)\alpha_1(z)-[\alpha_1(x),\rho_2^l(v)z]_1\nonumber\\ -&&\rho_2^r(\rho_1^r(z)v)\alpha_1(x)-\rho_2^r(\alpha_2(v))[x,z]_1\Big)+\Big( \rho_2^r(\alpha_2(w))[x,y]_1-[\alpha_1(x),\rho_2^r(w)y]_1\nonumber\\
 &&-\rho_2^r(\rho_1^l(y)w)\alpha_1(x)-
 [\rho_2^r(w)x,\alpha_1(y)]_1-\rho_2^l(\rho_1^l(x)w)\alpha_1(y)\Big)+
 \Big([\rho_1^l(x)v,\alpha_2(w)]_2\nonumber\\
 &&+\rho_1^l(\rho_2^r(v)x)\alpha_2(w)-\rho_1^l(\alpha_1(x))[v,w]_2-[\rho_1^l(x)w,\alpha_2(v)]_2-\rho_1^l(\rho_2^r(w)x)\alpha_2(v)\Big)\nonumber\\
 &&+\Big([\rho_1^r(y)u,\alpha_2(w)]_2+\rho_1^l(\rho_2^l(u)y)\alpha_2(w)-[\alpha_2(u),\rho_1^l(y)w]_2-\rho_1^r(\rho_2^r(w)y)\alpha_2(u)\nonumber\\
 &&-\rho_1^r(\alpha_1(y))[u,w]_2\Big)+\Big( \rho_1^r(\alpha_1(z)[u,v]_2-[\alpha_2(u),\rho_1^r(z)v]_2-\rho_1^r(\rho_2^l(v)z)\alpha_2(u)\nonumber\\
 &&-[\rho_1^r(z)u,\alpha_2(v)]_2-\rho_1^l(\rho_2^l(u)z)\alpha_2(v)\Big)=0
 \mbox{ ( by (\ref{mpHL1})-(\ref{mpHL6}) ).} \nonumber
\end{eqnarray}
\end{proof}
\begin{definition}
 Let $(V,\rho^l,\rho^r,\phi )$ be a representation of a Hom-Leibniz algebra $(A,[,],\alpha).$ A linear
operator $T : V\rightarrow A$ is called a relative Rota-Baxter operator on $(A,[,],\alpha)$ with respect to $(V,\rho^l,\rho^r,\phi )$ if $T$ satisfies
\begin{eqnarray}
 && T\phi=\alpha T \label{rbHL1}\\
  && [Tu,Tv]=T(\rho^l(Tu)+\rho^r(Tv)u) \mbox{ for all $u,v\in V.$} \label{rbHL2}
\end{eqnarray}
\end{definition}
Observe  that Rota-Baxter operators on Hom-Leibniz algebras are relative Rota-Baxter operators with
respect to the regular representation.
\begin{example}\label{exHLr}
 Consider the $2$-dimensional  Hom-Leibniz algebra 
$(A,[,],\alpha)$  where the non-zero products with respect to a basis $(e_1,e_2)$ are given by: $[e_1,e_2]=-[e_2,e_1]:=e_1$ and $\alpha(e_1):=-e_1,\ \alpha(e_2):=e_1+e_2.$\\
Then a linear map $T: A\rightarrow A$ defined by  $T(e_1):=a_{11}e_1+a_{21}e_2;\ 
T(e_2):=a_{12}e_1+a_{22}e_2$ is a relative Rota-Baxter on $(A,\cdot,\alpha)$ with respect to the regular representation  if and only if $T\alpha=\alpha T$ and
\begin{eqnarray}
[Te_i,Te_j]=T([Te_i, e_j]+[e_i,Te_j]) 
\mbox{ for all $i,j\in \{1,2\}.$}\label{exHLreq}
\end{eqnarray}
 The condition $\alpha T=T\alpha$ is equivalent to 
\begin{eqnarray}
 a_{21}=0;\  a_{11}+2a_{12}-a_{22}=0.\nonumber
\end{eqnarray}
For $i=j=1$ and  $i=j=2,$ the condition (\ref{exHLreq}) is satisfied trivially.
Similarly, we obtain for  $(i,j)\in\{(1,2),(2,1)\}$ 
\begin{eqnarray}
 a_{11}a_{22}=a_{11}a_{22}+a_{11}^2.\nonumber 
 \end{eqnarray} 
Summarize the above discussions, we observe
that $a_{11}=a_{21}=0, a_{22}=2a_{21}.$ Hence, the linear map $T: A\rightarrow A$ defined by $T(e_2):= a_{12}e_1+2a_{12}e_2$ is a  relative Rota-Baxter operator on $(A,[,],\alpha)$  with respect to the regular representation. 
\end{example}
\begin{lemma}\label{lHLRB}
 Let $T$ be a relative Rota-Baxter operator on a Hom-Leibniz algebra $(A,[,],\alpha)$ with respect to a representation  $(V,\rho^l,\rho^r,\phi ).$ If define a bracket $[,]_T$ on $V$ by 
 \begin{eqnarray}
  [u,v]_T:=\rho^l(Tu)v+\rho^r(Tv)u \mbox{ for all $(u, v)\in V^{\times 2}$} \label{opRB}
 \end{eqnarray}
then, $(V,[,]_T, \phi)$ is a Hom-Leibniz algebra.
\end{lemma}
\begin{proof}
 First, note that the multiplicativity of 
 $[,]_T$ with respect to $\phi$ follows from conditions (\ref{rHL1}) and (\ref{rbHL1}). Next, pick $u,v,w\in V$ and observe from (\ref{opRB}) and (\ref{rbHL2}) that $T[u,v]_T=[Tu,Tv].$ Therefore, we deduce:
 \begin{eqnarray}
  && [[u,v]_T,\phi(w)]_T=\rho^l([Tu,Tv])\phi(w)+\rho^l(T\phi(w))[u,v]_T,\nonumber\\                               
  &&=\rho^l([Tu,Tv])\phi(w)+
  \rho^r(T\phi(w))\rho^l(Tu)v+\rho^r(T\phi(w)\rho^r(Tv)u.\nonumber
 \end{eqnarray}
Similarly, we compute
\begin{eqnarray}
 && [\phi(u),[v,w]_T]_T=\rho^r([Tv,Tw])\phi(u)+\rho^l(T\phi(u))\rho^l(Tv)w+
 \rho^l(T\phi(u))\rho^l(Tw)v,\nonumber\\
 && [[u,w]_T,\phi(v)]_T=
 \rho^l([Tu,Tw])\phi(v)+
  \rho^r(T\phi(v))\rho^l(Tu)w+\rho^r(T\phi(v)\rho^r(Tw)u;\nonumber
\end{eqnarray}
Hence, the Hom-Leibniz identity in $(V,[,]_T,\phi)$ follows by (\ref{rbHL1}) and $(\ref{rHL2})-(\ref{rHL4}).$
\end{proof}
\begin{corollary}\label{cmorp}
 Let $T$ be a relative Rota-Baxter operator on a Hom-Leibniz algebra $(A,[,],\alpha)$ with respect to a representation  $(V,\rho^l,\rho^r,\phi ).$ Then, $T$ is a morphism from the Hom-Leibniz algebra 
 $(V,[,]_T,\phi)$ to the initial Hom-Leibniz algebra $(A,[,],\alpha).$
\end{corollary}
\begin{theorem}\label{tHLRB}
 Let $T$ be a relative Rota-Baxter operator on a Hom-Leibniz algebra $(A,[,],\alpha)$ with respect to a representation  $(V,\rho^l,\rho^r,\phi ).$  
 Then, $(A,\overline{\rho^l}, \overline{\rho^r},\alpha)$ is a representation of the Hom-Leibniz algebra
 $(V,[,]_T,\phi)$ where
 \begin{eqnarray}
 && \overline{\rho^l}(u)x:=[Tu,x]-T\rho^r(x)u   \mbox{ for all $(x,u)\in A\times V$}\label{rrb1}\\
 && \overline{\rho^r}(u)x:=[x,Tu]-T\rho^l(x)u  \mbox{ for all $(x,u)\in A\times V$}\label{rrb2}
 \end{eqnarray}
 \end{theorem}
 \begin{proof}
  First, observe that (\ref{rHL1}) in $(A,\overline{\rho^l}, \overline{\rho^r},\alpha)$ follows from the one in $(V,\rho^l,\rho^r,\phi ),$  the multiplicativity of $[,]$ with respect to $\alpha$ and (\ref{rbHL1}). Next, for all 
  $x\in A$ and $(u,v)\in V^{\times 2},$ if one oberves from Corollary \ref{cmorp} that $T[u,v]_T=[Tu,Tv],$ we compute:
  \begin{eqnarray}
   &&\overline{\rho^l}([u,v]_T)\alpha(x)=
   [[Tu,Tv],\alpha(x)]-T\rho^r(\alpha(x))\rho^l(Tu)v-T\rho^r(\alpha(x))\rho^r(Tv)u \mbox{ ( by (\ref{rrb1}) )}\nonumber\\
   &&=[[Tu,Tv],\alpha(x)]-T\rho^l(T\phi(u))\rho^r(x)v-T\rho^l([Tu,x])\phi(v)
   -T\rho^r([Tv,x])\phi(u) \nonumber\\
   &&- T\rho^r(T\phi(v))\rho^r(x)u 
   \mbox{ ( by (\ref{rHL3}) and (\ref{rHL4}) in $(V,\rho^l,\rho^r,\phi )$ ). }\nonumber
  \end{eqnarray}
Also, we obtain by (\ref{rrb1}) and (\ref{rrb2}):
\begin{eqnarray}
 &&\overline{\rho^l}(\phi(u))\overline{\rho^l}(v)x=[T\phi(u),[Tv,x]]-[T\phi(u),T\rho^r(x)v]-T\rho^r([Tv,x])\phi(u)+T\rho^r(T\rho^r(x)v)\phi(u),\nonumber\\
 &&\overline{\rho^r}(\phi(v))\overline{\rho^l}(u)x=[[Tu,x],T\phi (v)]-[T\rho^r(x)u, T\phi(v)]-T\rho^l([Tu,x])\phi(v)+T\rho^l(T\rho^r(x)u)\phi(v).\nonumber 
\end{eqnarray}
Using $T\phi=\alpha T$ and
(\ref{leib}), we obtain by (\ref{rbHL2}):
\begin{eqnarray}
&& \overline{\rho^l}([u,v]_T)\alpha(x)-\overline{\rho^l}(\phi(u))\overline{\rho^l}(v)x-\overline{\rho^r}(\phi(v))\overline{\rho^l}(u)x=
\Big([T\phi(u),T\rho^r(x)v]-T\rho^l(T\phi(u))\rho^r(x)v\nonumber\\
&&-T\rho^r(T\rho^r(x)v)\phi(u)\Big)+
\Big([T\rho^r(x)u,T\phi(v)]-T\rho^l(T\rho^r(x)u)\phi(v)-T\rho^r(T\phi(v))\rho^r(x)u\Big)=0.\nonumber
\end{eqnarray}
Hence, (\ref{rHL2}) holds in $(A,\overline{\rho^l}, \overline{\rho^r},\alpha).$ Now, to prove (\ref{rHL3}) in 
$(A,\overline{\rho^l}, \overline{\rho^r},\alpha),$ by straightforward computations using (\ref{rrb1}) and (\ref{rrb2}) we obtain:
\begin{eqnarray}
 &&\overline{\rho^r}(\phi(v))\overline{\rho^l}(u)x=[[Tu,x],T\phi(v)]-[T\rho^r(x)u,T\phi(v)]-T\rho^l([Tu,x])\phi(v)+T\rho^l(T\rho^r(x)u)\phi(v)\nonumber\\
 &&=[T\phi(u),[x,Tv]]+[[Tu,Tv],\alpha(x)]-[T\rho^r(x)u,T\phi(v)]-T\rho^l(T\phi(u))\rho^l(x)v-T\rho^r(\alpha(x)\rho^l(Tu)v\nonumber\\
 &&+ T\rho^l(T\rho^r(x)u)\phi(v) \mbox{ ( by (\ref{rbHL1}), (\ref{leib}) and (\ref{rHL2}) ).}\nonumber
\end{eqnarray}
Similarly, we get:
\begin{eqnarray}
 &&\overline{\rho^l}(\phi(u))\overline{\rho^r}(v)x=[T\phi(u),[x,Tv]]-[T\phi(u),T\rho^l(x)v]-T\rho^r([x,Tv])\phi(u)+T\rho^r(T\rho^l(x)v)\phi(u),\nonumber\\
 &&\overline{\rho^l}([u,v]_T)\alpha(x)=
   [[Tu,Tv],\alpha(x)]-T\rho^r(\alpha(x))\rho^l(Tu)v-T\rho^r(\alpha(x))\rho^r(Tv)u \nonumber\\
   &&=[[Tu,Tv],\alpha(x)]-T\rho^r(\alpha(x))\rho^l(Tu)v
   -T\rho^r([Tv,x])\phi(u) 
- T\rho^r(T\phi(v))\rho^r(x)u 
   \mbox{ ( by (\ref{rHL4})  ). }\nonumber
\end{eqnarray}
Hence, after rearranging terms we come to
\begin{eqnarray}
 &&\overline{\rho^r}(\phi(v))\overline{\rho^l}(u)x-\overline{\rho^l}(\phi(u))\overline{\rho^r}(v)x-\overline{\rho^l}([u,v]_T)\alpha(x)
 =\Big(-[T\rho^r(x)u,T\phi(v)]+T\rho^l(T\rho^r(x)u)\phi(v)\nonumber\\
 &&+T\rho^r(T\phi(v))\rho^r(x)u \Big) 
 +\Big([T\phi(u),T\rho^l(x)v]-T\rho^l(T\phi(u))\rho^l(x)v-T\rho^r(T\rho^l(x)v)\phi(u) \Big)\nonumber\\
 &&+ 
 \Big(T\rho^r([x,Tv])\phi(v)+T\rho^r([Tv,x])\phi(v)\Big)=0  \mbox{ (  by (\ref{rbHL2}) and (\ref{rHL4r}) ).}\nonumber
\end{eqnarray}
Finally, to prove (\ref{rHL4}), we first compute using (\ref{rrb2}):
\begin{eqnarray}
 &&\overline{\rho^r}(\phi(v))\overline{\rho^r}(u)x=[[x,Tu],T\phi(v)]-
 [T\rho^l(x)u,T\phi(v)]-T\rho^l([x,Tu])\phi(v)+T\rho^l(T\rho^l(x)u)\phi(v)\nonumber\\
 &&=[\alpha(x),[Tu,Tv]]+[[x,Tv],T\phi(u)]-
 [T\rho^l(x)u,T\phi(v)]-T\rho^l(\alpha(x))\rho^l(Tu)v-T\rho^l(T\phi(u))\rho^l(x)v\nonumber\\
 && +T\rho^l(T\rho^l(x)u)\phi(v) \mbox{ ( by (\ref{rbHL1}), (\ref{leib}) and (\ref{rHL2}) ),}\nonumber\\
 &&\overline{\rho^r}([u,v]_T)\alpha(x)=[\alpha(x),T[u,v]_T]-T\rho^l(\alpha(x))\rho^l(Tu)v-T\rho^l(\alpha(x))\rho^r(Tv)u\nonumber\\
 &&=[\alpha(x),[Tu,Tv]]-T\rho^l(\alpha(x))\rho^l(Tu)v-T\rho^r(T\phi(u))\rho^l(x)u+T\rho^l([x,Tv])\phi(u) 
 \nonumber\\
 &&\mbox{ ( by (\ref{rbHL1}), (\ref{leib}) and (\ref{rHL3}) ),}\nonumber\\
 &&\overline{\rho^r}(\phi(u))\overline{\rho^r}(v)x=[[x,Tv],T\phi(u)]-
 [T\rho^l(x)v,T\phi(u)]-T\rho^l([x,Tv])\phi(u)+T\rho^l(T\rho^l(x)v)\phi(u).\nonumber
\end{eqnarray}
Therefore using these equations, we come to
\begin{eqnarray}
 &&\overline{\rho^r}(\phi(v))\overline{\rho^r}(u)x-\overline{\rho^r}([u,v]_T)\alpha(x)-\overline{\rho^r}(\phi(u))\overline{\rho^r}(v)x=\Big(-[T\rho^l(x)u,T\phi(v)]+T\rho^l(T\rho^l(x)u)\phi(v)\nonumber\\&&+T\rho^r(T\phi(v))\rho^l(x)u \Big)
 +\Big(T\rho^l(x)v,T\phi(u)]-T\rho^l(T\rho^l(x)v)\phi(u)- T\rho^r(T\phi(u))\rho^l(x)v\Big)=0 \mbox{ ( by (\ref{rbHL2})  ). }\nonumber
\end{eqnarray}
 \end{proof}
\section{Representations and relative Rota-Baxter operators of Hom-Leibniz-Poisson algebras.}
This section contains our interesting results. Here, we introduce and study  representations of Hom-Leibniz Poisson algebras. The notions of matched pairs and relative Rota-Baxter operators  of these Hom-algebras are also treated to be in adequacy with the preceding sections. Most of the proofs are complements to the proofs performed in the previous sections.
 \begin{definition}
 A representation of a Hom-Leibniz Poisson algebra $(A, \cdot, [,], \alpha)$  is a sextuple $(V, \lambda^l,\lambda^r ,\rho^l,\rho^r,\phi)$ where $V$ is a vector space, $\phi\in gl(V)$ and $\rho^l, \rho^r, \lambda^l,\lambda^r: A\rightarrow gl(V)$ are five linear maps
such that 
\begin{enumerate}
\item $(V,\lambda^l,\lambda^r,\phi)$ is a representation of the Hom-associative algebra 
 $(A,\cdot,\alpha).$
 \item $(V,\rho^l,\rho^r,\phi)$ is a representation of the Hom-Leibniz algebra 
 $(A,[,],\alpha).$
 \item the following equalities hold for all $x, y \in A:$
 \begin{eqnarray}
 \rho^r(\alpha(y)\lambda^l(x)
 =\lambda^l(\alpha(x))\rho^r(y)+\lambda^l([x,y])\phi\label{rHLP1}\\
 \rho^r(\alpha(y))\lambda^r(x)=\lambda^r([x,y])\phi+\lambda^r(\alpha(x))\rho^r(y)\label{rHLP2}\\
\rho^l(x\cdot y)\phi=\lambda^l(\alpha(x))\rho^l(y)+\lambda^r(\alpha(y))\rho^l(x)\label{rHLP3}
\end{eqnarray}
\end{enumerate}
\end{definition}
As a generalization of Proposition \ref{PEx1} and Proposition \ref{PEx2}, the following result allows to give examples of representations of Hom-Leibniz Poisson algebras.
\begin{proposition}
 Let $\mathcal{A}_1:=(A_1,\mu_1, [,]_1,\alpha_1)$ and $\mathcal{A}_2:=(A_2,\mu_2,[,]_2,\alpha_2)$ be two Hom-Leibniz Poisson algebras and $f:\mathcal{A}_1\rightarrow \mathcal{A}_2$ be a morphism of Hom-Leibniz Poisson algebras. Then $A_2^f:=(A_2, \lambda^l,\lambda^r, \rho^l,\rho^r,\alpha_2)$ is a representaion of $\mathcal{A}_1$ where $\lambda^l(a)b:=\mu_2(f(a),b),\ \lambda^r(a,b):=\mu_2(b,f(a)),\ 
 \rho^l(a)b:=[f(a),b]_2,\ \rho^r(a,b):=[b,f(a)]_2$ for all $(a,b)\in A\times B.$
\end{proposition}
\begin{proof}
 Thanks to Proposition \ref{PEx1} and Proposition \ref{PEx2}, $(A_2, \lambda^l,\lambda^r,\alpha_2)$  and $(A_2,  \rho^l,\rho^r,\alpha_2)$ are representations of the Hom-associative algebra $(A_1, \mu_1,\alpha_1)$ and the Hom-Leibniz algebra $(
 A_1,\alpha_1,[,]_1)$ respectively. It remains to prove (\ref{rHLP1})-(\ref{rHLP3}). Let $x,y\in A_1$ and $z\in A_2.$ Since $f$ is a morphism, using (\ref{cHLP}) in $\mathcal{A}_2$, we get
 \begin{eqnarray}
 && \rho^r(\alpha_1(y))\lambda^l(x)z=
  [\mu_2(f(x),z),f\alpha_1(y)]_2=
  [\mu_2(f(x),z),\alpha_2f(y)]_2=
  \mu_2(\alpha_2f(x),[z,f(y)]_2\nonumber\\
  &&+\mu_2([f(x),f(y)]_2,\alpha_2(z))=
  \mu_2(f\alpha_1(x),[z,f(y)]_2
  +\mu_2(f([x,y]_1),\alpha_2(z))=\lambda^l(\alpha_1(x))\rho^r(y)z\nonumber\\
  &&+\lambda^l([x,y]_1)\alpha_2(z). \nonumber  
 \end{eqnarray}
Hence, (\ref{rHLP1}) holds and similarly, (\ref{rHLP2}) is obtained. Finally, by same hypothesis, we have
\begin{eqnarray}
 &&\rho^l(\mu_1(x,y))\alpha_2(z)=[f\mu_1(x,y),\alpha_2(z)]_2=
 [\mu_2(f(x),f(y)),\alpha_2(z)]_2
 =\mu_2(\alpha_2f(x),[f(y),z]_2)\nonumber\\
  &&+\mu_2([f(x),z]_2,\alpha_2f(y))=
  \mu_2(f\alpha_1(x),[f(y),z]_2)
  +\mu_2([f(x),z]_2),f\alpha_1(y))
  =\lambda^l(\alpha_1(x))\rho^l(y)z\nonumber\\
  &&+\lambda^r(\alpha_1(y))\rho^l(x)z\nonumber
\end{eqnarray}
which means that (\ref{rHLP3}) holds.
\end{proof}
\begin{example}
\begin{enumerate}
 \item Let $(A, \cdot, [,], \alpha)$ be a Hom-Leibniz Poisson algebra.
 Define left multiplications 
 $L, l: A\rightarrow gl(A)$ and
  right multiplications $R,r: A\rightarrow gl(A)$ by $L_x y:=[x, y];\  l_xy:=x\cdot y$ and $R_x y:= [y,x];\  r_xy:=y\cdot x$  for all $x, y \in A.$ Then $(A, l,r, L, R,\alpha)$ is a
representation of $(A, \cdot, [,], \alpha),$ which is called a regular representation.
 \item Let $(A,\cdot,[,],\alpha)$ be a Hom-Leibniz Poisson  algebra and $(B,\alpha)$ de a two-sided Hom-ideal of $(A,\cdot,[,],\alpha).$ Then $(B,\alpha)$  inherits a structure of representation of $(A,\cdot,[,],\alpha)$ where $\lambda^l(a)b:=a\cdot b; \lambda^r(a)b:=b\cdot a$ and
 $\rho^l(a)b:=[a,b];\ \rho^r(b,a):=[b,a]$ for all $(a,b)\in A\times B.$
\end{enumerate}
\end{example}
\begin{proposition}
 Let $\mathcal{V}:=(V, \lambda^l,\lambda^r,\rho^l,\rho^r,\phi)$ be a representation of a Hom-Leibniz Poisson algebra $\mathcal{A}:=(A,\cdot,[,],\alpha)$ and $\beta$ be a self-morphism of $\mathcal{A}=(A, \cdot, [,],\alpha)$. Then 
 $\mathcal{V}_{\beta}=(V,    
 \lambda^l_{\beta}:=\lambda^l\beta,\lambda^r_{\beta}:=\lambda^r\beta,
 \rho^l_{\beta}:=\rho^l\beta,\rho^r_{\beta}:=\rho^r\beta,\phi)$ is a representation of $\mathcal{A}.$ 
\end{proposition}
\begin{proof}
 We know that $(V,    
 \lambda^l_{\beta}:=\lambda^l\beta,\lambda^r_{\beta}:=\lambda^r\beta,\phi)$ is a representation of the Hom-associative algebra $(A, \cdot,\alpha)$  by Proposition \ref{sRHas}  and
  $(V,\rho^l_{\beta}:=\rho^l\beta,\rho^r_{\beta}:=\rho^r\beta,\phi)$ is a representation of the Hom-Leibniz algebra
  $(A, [,],\alpha)$ thanks to Proposition \ref{sRHLeib}. It remains to prove (\ref{rHLP1})-(\ref{rHLP3}). Now, for all $x,y\in A,$ since $\beta$ is a morphism, we obtain by (\ref{rHLP1}), 
   and (\ref{rHLP3}) in 
  $\mathcal{V}$ respectively
  \begin{eqnarray}
&&\rho^r_{\beta}(\alpha(y))\lambda^l_{\beta}(x)=\rho^r(\beta\alpha(y))\lambda^l(\beta(x))=\rho^r(\alpha\beta(y))\lambda^l(\beta(x))=\lambda^l(\alpha\beta(x))\rho^r(\beta(y)),\nonumber\\
&&+\lambda^l([\beta(x),\beta(y)])\phi=
\lambda^l(\beta\alpha(x))\rho^r(\beta(y))
+\lambda^l(\beta([x,y]))\phi
=\lambda^l_{\beta}(\alpha(x))\rho^r_{\beta}(y)+\lambda^l_{\beta}([x,y])\phi,\nonumber\\
 &&\rho^l_{\beta}(x\dot y)\phi=\rho^l(\beta(x)\cdot\beta(y))\phi=
 \lambda^l(\alpha\beta(x))\rho^l(\beta(y))+
 \lambda^r(\alpha\beta(y))\rho^l(\beta(x))=
\lambda^l(\beta\alpha(x))\rho^l(\beta(y))
\nonumber\\
 &&
 +\lambda^r_{\beta}(\alpha(y))\rho^l_{\beta}(x)
=\lambda^r_{\beta}(\alpha(y))\rho^l_{\beta}(x).\nonumber
 \end{eqnarray}
Hence, we get (\ref{rHLP1})   and (\ref{rHLP3}) for $\mathcal{V}_{\beta}.$ As for (\ref{rHLP1}), we can also check that (\ref{rHLP2}) is satisfied for $\mathcal{V}_{\beta}.$
\end{proof}

\begin{corollary}
 Let $\mathcal{V}=(V, \lambda^l,\lambda^r ,\rho^l,\rho^r,\phi)$ be a representation of a Hom-Leibniz Poisson algebra $\mathcal{A}=(A,\cdot,[,],\alpha).$ Then 
 $\mathcal{V}^{(n)}=(V, \lambda^l\alpha^n,\lambda^r\alpha^n ,\rho^l\alpha^n,\rho^r\alpha^n,\phi)$ is a representation of $\mathcal{A}$  for each $n\in\mathbb{N}.$
\end{corollary}
\begin{proposition}
 Let $(A,\cdot,[,],\alpha)$ be a Hom-Leibniz Poisson algebra and let $(V,\phi)$ be a Hom-module and $\lambda^l,\lambda^r, \rho^l,\rho^r : A\rightarrow gl(V)$ be four linear maps. Then, the sextuple $(V, \lambda^l,\lambda^r,\rho^l,\rho^r,\phi )$ is a representation of $(A,\cdot,[,],\alpha)$  if and only if the direct sum of vector spaces, $A\oplus V,$  turns into a Hom-Leibniz Poisson
algebra with the multiplications defined on $A\oplus V$ by
\begin{eqnarray}
 (x+u)\ast(y+v):=x\cdot y+(\lambda^l(x)v+\lambda^r(y)u)\label{sdp1}\\
  \{(x+u),(y+v)\}:=[x,y]+(\rho^l(x)v+\rho^r(y)u)\label{sdp2}\\
 (\alpha\oplus\phi)(x+u):=\alpha(x)+\phi(u)\label{sdp3}
\end{eqnarray}
\end{proposition}
\begin{proof}  Thanks to proofs of Proposition \ref{spHAs} and Proposition \ref{spHLa}, it suffices to prove that $(V, \lambda^l,\lambda^r,\rho^l,\rho^r,\phi )$ satisfies (\ref{rHLP1}), (\ref{rHLP2}) and (\ref{rHLP3}) if and only if  $(A,\cdot,[,],\alpha)$ satisfies (\ref{cHLP}). Using the definition of $\ast$ and $\{,\},$ we have for any $(x,y)\in A^{\times 2},\ (u,v)\in V^{\times 2}:$
 \begin{eqnarray}
  &&\{(x+u)\ast (y+v),(\alpha\oplus)(z+w)\}
  \nonumber\\
  &&=\{x\cdot y+\lambda^l(x)v+\lambda^r(y)u, \alpha(z)+\phi(w)\}\nonumber\\
  &&=[x\cdot y,\alpha(z)]+\rho^l(x\cdot y)\phi(w)+\rho^r(\alpha(z))\lambda^l(x)v+\rho^r(\alpha(z))\lambda^r(y)u\nonumber
  %&&=\alpha(x)\cdot[y,z]+[x,z]\cdot\alpha(y)+\lambda^l(\alpha(x))\rho^l(y)w+\lambda^r(\alpha(y))\rho^l(x)w\nonumber\\
  %&&+\lambda^l(\alpha(x))\rho^r(z)v+\lambda^l([x,z])\phi(v)+\lambda^r([y,z])\phi(u)+\lambda^r(\alpha(y))\rho^r(z)u\nonumber\\
  %&&( \mbox{by (\ref{rHLP1}), (\ref{rHLP2}),(\ref{rHLP3})} )\nonumber\\
  \end{eqnarray}
  and similarly,
  \begin{eqnarray}
  &&(\alpha\oplus\phi)(x+u)\ast\{(y+v),(z+w)\}\nonumber\\
  &&=\big(\alpha(x)\cdot[y,z]+\lambda^l(\alpha(x))\rho^l(y)w
  +\lambda^l(\alpha(x))\rho^r(z)v+
  +\lambda^r([y,z])\phi(u)\big),\nonumber\\
  &&\{(x+u),(z+w)\}\ast(\alpha\oplus\phi)(y+v)\nonumber\\
  &&=\big([x,z]\cdot\alpha(y)+\lambda^l([x,z])\phi(v)+\lambda^r(\alpha(y))\rho^l(x)w
  +\lambda^r(\alpha(y))\rho^r(z)u\big).\nonumber
  %&&=\big(\alpha(x)\cdot[y,z]+\lambda^l(\alpha(x))\rho^l(y)w
  %+\lambda^l(\alpha(x))\rho^r(z)v+
  %+\lambda^r([y,z])\phi(u)\big)\nonumber\\
  %&&+\big([x,z]\cdot\alpha(y)+\lambda^l([x,z])\phi(v)+\lambda^r(\alpha(y))\rho^l(x)w
  %+\lambda^r(\alpha(y))\rho^r(z)u\big)\nonumber\\
  %&&(\alpha\oplus\phi)(x+u)\ast\{(y+v),(z+w)\}
  %+\{(x+u),(z+w)\}\ast(\alpha\oplus\phi)(y+v)\nonumber
 \end{eqnarray}
 Hence by (\ref{leib}), we get:
 \begin{eqnarray}
  &&\{(x+u)\ast (y+v),(\alpha\oplus)(z+w)\}-
  (\alpha\oplus\phi)(x+u)\ast\{(y+v),(z+w)\}\nonumber\\
  &&-\{(x+u),(z+w)\}\ast(\alpha\oplus\phi)(y+v)=\Big(\rho^l(x\cdot y)\phi(w)-\lambda^l(\alpha(x))\rho^l(y)w-\lambda^r(\alpha(y))\rho^l(x)w\Big)\nonumber\\
  &&+\Big(\rho^r(\alpha(z))\lambda^l(x)v-\lambda^l(\alpha(x))\rho^r(z)v-\lambda^l([x,z])\phi(v) \Big)+ \Big(\rho^r(\alpha(z))\lambda^r(y)u-\lambda^r([y,z])\phi(u)\nonumber\\
  &&-\lambda^r(\alpha(y))\rho^r(z)u \Big).\nonumber
 \end{eqnarray}
 Therefore, (\ref{rHLP1}), (\ref{rHLP2}) and (\ref{rHLP3}) hold in $(V, \lambda^l,\lambda^r,\rho^l,\rho^r,\phi )$  if and only if (\ref{cHLP}) holds in $(A,\cdot,[,],\alpha)$.
\end{proof}
\begin{definition}
 Let $\mathcal{A}_1:=(A_1,\cdot, [,]_1,\alpha_1)$ and 
 $\mathcal{A}_2:=(A_2,\bullet, [,]_2,\alpha_2)$ be two Hom-Leibniz Poisson algebras. Let $\lambda_1^l, \lambda_1^r,\rho_1^l,\rho_1^r: A_1\rightarrow gl(A_2)$ and 
 $\lambda_2^l, \lambda_2^r,\rho_2^l,\rho_2^r: A_2\rightarrow gl(A_1)$ be linear maps such that $(A_2,\lambda_1^l, \lambda_1^r,\rho_1^l,\rho_1^r, \alpha_2)$ is a representation of $\mathcal{A}_1,$ 
 $(A_1,\lambda_2^l, \lambda_2^r,\rho_2^l,\rho_2^r, \alpha_1)$ is a representation of $\mathcal{A}_2$ such that
 \begin{enumerate}
  \item  $((A_1,\lambda_2^l,\lambda_2^r, \alpha_1),(A_2,\lambda_1^l,\lambda_1^r, \alpha_2))$ is a matched pair of Hom-associative algebras,
  \item $((A_1,\rho_2^l,\rho_2^r, \alpha_1),(A_2,\rho_1^l,\rho_1^r, \alpha_2))$ is a matched pair of Hom-Leibniz algebras,
  \item the following identities
  \begin{eqnarray}
  &&\lambda_2^l(\alpha_2(u))[x,y]_1+(\rho_2^l(u)y)\cdot\alpha_1(x)+\lambda_2^l(\rho_1^r(y)u)\alpha_1(x)-
  [\lambda_2^l(u)x,\alpha_1(y)]_1
  \nonumber\\
  &&-\rho_2^l(\lambda_1^r(x)u)\alpha_1(y)=0\label{mpHLP1}\\
  &&\lambda_2^r(\alpha_2(u))[x,y]_1+\alpha_1(x)\cdot(\rho_2^l(u)y)+\lambda_2^r(\rho_1^r(y)u)\alpha_1(x)
  -[\lambda_2^r(u)x,\alpha_1(y)]_1\nonumber\\
  &&-\rho_2^l(\lambda_1^l(x)u)\alpha_1(y)=0\label{mpHLP2}\\
  &&\lambda_1^l(\alpha_1(x))[u,v]_2+(\rho_1^l(x)v)\bullet\alpha_2(u)+\lambda_1^l(\rho_2^r(v)x)\alpha_2(u)-
  [\lambda_1^l(x)u,\alpha_2(v)]_2
  \nonumber\\
  &&-\rho_1^l(\lambda_2^r(u)x)\alpha_2(v)=0\label{mpHLP3}\\
  &&\lambda_1^r(\alpha_1(x))[u,v]_2+\alpha_2(u)\bullet(\rho_1^l(x)v)+\lambda_1^r(\rho_2^r(v)x)\alpha_2(u)
  -[\lambda_1^r(x)u,\alpha_2(v)]_2\nonumber\\
  &&-\rho_1^l(\lambda_2^l(u)x)\alpha_2(v)=0\label{mpHLP4}\\
  &&\rho_2^r(\alpha_2(u))(x\cdot y)-\alpha_1(x)\cdot(\rho_2^r(u)y)-\lambda_2^r(\rho_1^l(y)u)\alpha_1(x)-
  (\rho_2^r(u)x)\cdot\alpha_1(y)\nonumber\\
  &&-\lambda_2^l(\rho_1^l(x)u)\alpha_1(y)=0\label{mpHLP5}\\
  &&\rho_1^r(\alpha_1(x))(u\bullet v)-\alpha_2(u)\bullet(\rho_1^r(x)v)-\lambda_1^r(\rho_2^l(v)x)\alpha_2(u)-
  (\rho_1^r(x)u)\bullet\alpha_2(v)\nonumber\\
  &&-\lambda_1^l(\rho_2^l(u)x)\alpha_2(v)=0\label{mpHLP6}
 \end{eqnarray}
 hold for all $x,y,z\in A_1,$ $u,v,w\in A_2.$
 \end{enumerate}
Then $A^{\Join}:=((A_1,\lambda_2^l, \lambda_2^r,\rho_2^l,\rho_2^r, \alpha_1),(A_2,\lambda_1^l, \lambda_1^r,\rho_1^l,\rho_1^r, \alpha_2))$ is called a matched pair of Hom-Leibniz Poisson algebras.
\end{definition}
\begin{proposition}
 Let $((A_1,\lambda_2^l, \lambda_2^r,\rho_2^l,\rho_2^r, \alpha_1),(A_2,\lambda_1^l, \lambda_1^r,\rho_1^l,\rho_1^r, \alpha_2))$ be a matched pair of Hom-Leibniz Poisson
algebra. Then, there is a Hom-Leibniz Poisson algebra\\ $(A_1\oplus A_2, \ast,
\{,\},\alpha_1\oplus\alpha_2)$ defined by
\begin{eqnarray}
 (x+u)\ast(y+v)&=&(x\cdot y+\lambda_2^l(u)y+\lambda_2^r(v)x)+(u\bullet v+\lambda_1^l(x)v+\lambda_1^r(y)u) \label{op1}\\
 \{(x+u),(y+v)\}&=&([x,y]_1+\rho_2^l(u)y+\rho_2^r(v)x)+([u,v]_2+\rho_1^l(x)v+\rho_1^r(y)u) \label{op2}\\
 (\alpha_1\oplus\alpha_2)(x+u)&=&\alpha_1(x)+\alpha_2(u)\nonumber
\end{eqnarray}
\end{proposition}
\begin{proof}
 It is clear that $(A_1\oplus A_2, \ast,\alpha_1\oplus\alpha_2)$ is a Hom-associative algebra (see Proposition \ref{spHAs}) and $(A_1\oplus A_2, \{,\},\alpha_1\oplus\alpha_2)$ is a Hom-Leibniz algebra (see Proposition \ref{spHLa}). It remains to prove (\ref{cHLP}). Pick 
 $x,y,z\in A_1$ and $u,v,w\in A_2.$ Then, by a straightforward computation, we get
 \begin{eqnarray}
  &&\{ (x+u)\ast(y+v), (\alpha_1\oplus\alpha_2)(z+w)\}
  \nonumber\\
  &&=
  \underbrace{[x\cdot y,\alpha_1(z)]_1}_{(j_1)}+[\lambda_2^l(u)y,\alpha_1(z)]_1
  +[\lambda_2^r(v)x,\alpha_1(z)]_1\nonumber\\
  &&+\underbrace{\rho_2^l(u\bullet v)\alpha_1(z)}_{(j_2)}+
\rho_2^l(\lambda_1^l(x)v)\alpha_1(z)+\rho_2^l(\lambda_1^r(y)u)\alpha_1(z)\nonumber\\
&&+\rho_2^r(\alpha_2(w))(x\cdot y)+\underbrace{\rho_2^r(\alpha_2(w))\lambda_2^l(u)y}_{(j_3)}+\underbrace{\rho_2^r(\alpha_2(w))\lambda_2^r(v)x}_{(j_4)}\nonumber\\
&&+\underbrace{[u\bullet v,\alpha_2(w)]_2}_{(j_1)}+[\lambda_1^l(x)v,\alpha_2(w)]_2
  +[\lambda_1^r(y)u,\alpha_2(w)]_2\nonumber\\
  &&+\underbrace{\rho_1^l(x\cdot y)\alpha_2(w)}_{(j_2)}+
\rho_1^l(\lambda_2^l(u)y)\alpha_2(w)+\rho_1^l(\lambda_2^r(v)x)\alpha_2(w)\nonumber
\end{eqnarray}
\begin{eqnarray}
&&+\rho_1^r(\alpha_1(z))(u\bullet v)+\underbrace{\rho_1^r(\alpha_1(z))\lambda_1^l(x)v}_{(j_3)}+\underbrace{\rho_1^r(\alpha_1(z))\lambda_1^r(y)u}_{(j_4)},\nonumber\\
&&(\alpha_1\oplus\alpha_2)(x+u)\ast\{y+v,z+w\}\nonumber\\
&&=\alpha_1(x)\cdot[y,z]_1+\alpha_1(x)\cdot(\rho_2^l(v)z)+\alpha_1(x)\cdot(\rho_2^r(w)y)\nonumber\\
&&+\lambda_2^l(\alpha_2(u))[y,z]_1+\lambda_2^l(\alpha_2(u))\rho_2^l(v)z+\lambda_2^l(\alpha_2(u))\rho_2^r(w)y\nonumber\\
&&+\lambda_2^r([v,w]_2)\alpha_1(x)+\lambda_2^r(\rho_1^l(y)w)\alpha_1(x)+\lambda_2^r(\rho_1^r(z)v)\alpha_1(x)\nonumber\\
&&+\alpha_2(u)\bullet[v,w]_1+\alpha_2(u)\bullet(\rho_1^l(y)w)+\alpha_2(u)\bullet(\rho_1^r(z)v)\nonumber\\
&&+\lambda_1^l(\alpha_1(x))[v,w]_2+\lambda_1^l(\alpha_1(x))\rho_1^l(y)w+\lambda_1^l(\alpha_1(x))\rho_1^r(z)v\nonumber\\
&&+\lambda_1^r([y,z]_1)\alpha_2(u)+\lambda_1^r(\rho_2^l(v)z)\alpha_2(u)+\lambda_1^r(\rho_2^r(w)y)\alpha_2(u),\nonumber\\
&&\{x+u,y+v\}\ast(\alpha_1\oplus\alpha_2)(y+v)\nonumber\\
&&=[x,z]_1\cdot\alpha_1(y)+(\rho_2^l(u)z)\cdot\alpha_1(y)+(\rho_2^r(w)x)\cdot\alpha_1(y)\nonumber\\
&&+\lambda_2^l([u,w]_2)\alpha_1(y)+\lambda_2^l(\rho_1^l(x)w)\alpha_1(y)+\lambda_2^l(\rho_1^r(z)u)\alpha_1(y)\nonumber\\
&&+\lambda_2^r(\alpha_2(v))[x,z]_1+\lambda_2^r(\alpha_2(v))\rho_2^l(u)z+\lambda_2^r(\alpha_2(v))\rho_2^r(w)x\nonumber\\
&&+[u,w]_2\bullet\alpha_2(v)+(\rho_1^l(x)w)\bullet\alpha_2(v)+(\rho_1^r(z)u)\bullet\alpha_2(v)\nonumber\\
&&+\lambda_1^l([x,z]_1)\alpha_2(v)+\lambda_1^l(\rho_2^l(u)z)\alpha_2(v)+\lambda_1^l(\rho_2^r(w)x)\alpha_2(v)\nonumber\\
&&+\lambda_1^r(\alpha_1(y))[u,w]_2+\lambda_1^r(\alpha_1(y))\rho_1^l(x)w+\lambda_1^r(\alpha_1(y))\rho_1^r(z)u.\nonumber
 \end{eqnarray}
 Therefore, applying $(\ref{cHLP}), (\ref{rHLP3}),(\ref{rHLP1}), (\ref{rHLP2})$ for $(j_1), (j_2), (j_3), (j_4)$ respectively, we get:
 \begin{eqnarray}
  &&\{ (x+u)\ast(y+v), (\alpha_1\oplus\alpha_2)(z+w)\}-
  (\alpha_1\oplus\alpha_2)(x+u)\ast\{y+v,z+w\}\nonumber\\
  &&-\{x+u,y+v\}\ast(\alpha_1\oplus\alpha_2)(y+v)\nonumber\\
  &&=\Big([\lambda_2^l(u)y,\alpha_1(z)]_1+\rho_2^l(\lambda_1^r(y)u)\alpha_1(z)-(\rho_2^l(u)z)\cdot\alpha_1(y)-\lambda_2^l(\rho_1^r(z)u)\alpha_1(y)\nonumber\\
  &&-\lambda_2^l(\alpha_2(u))[y,z]_1\Big)+\Big([\lambda_2^r(v)x,\alpha_1(z)]_1+\rho_2^l(\lambda_1^l(x)v)\alpha_1(z)-(\lambda_2^r(\rho_1^r(z)v)\alpha_1(x)-\nonumber\\
  &&-\alpha_1(x)\cdot(\rho_2^l(v)z)-\lambda_2^r(\alpha_2(v))[x,z]_1 \Big)+
  \Big([\lambda_1^l(x)v,\alpha_2(w)]_2+\rho_1^l(\lambda_2^r(v)x)\alpha_2(w)-
  \nonumber\\
  &&(\rho_1^l(x)w)\bullet\alpha_2(v)-\lambda_1^l(\rho_2^r(w)x)\alpha_2(v)
  -\lambda_1^l(\alpha_1(x))[v,w]_2\Big)+
  \Big([\lambda_1^r(y)u,\alpha_2(w)]_2\nonumber\\
  &&+\rho_1^l(\lambda_2^l(u)y)\alpha_2(w)-
  \lambda_1^r(\rho_2^r(w)y)\alpha_2(u)-\alpha_2(u)\bullet(\rho_1^l(y)w)-\lambda_1^r(\alpha_1(y))[u,w]_2\Big)\nonumber\\
  &&\Big(\rho_2^r(\alpha_2(w))(x\cdot y)-\alpha_1(x)\cdot(\rho_2^r(w)y)-\lambda_2^r(\rho_1^l(y)w)\alpha_1(x)-(\rho_2^r(w)x)\cdot\alpha_1(y)\nonumber\\
  &&-\lambda_2^l(\rho_1^l(x)w)\alpha_1(y)\Big)+
  \Big(\rho_1^r(\alpha_1(z))(u\bullet v)-\alpha_2(u)\bullet(\rho_1^r(z)v)-\lambda_1^r(\rho_2^l(v)z)\alpha_2(u)\nonumber\\
  &&-(\rho_1^r(z)u)\bullet\alpha_2(v)-\lambda_1^l(\rho_2^l(u)z)\alpha_2(v)\Big)
=0 \mbox{ ( by  (\ref{mpHLP1})-(\ref{mpHLP6}) ).}\nonumber
  \end{eqnarray}
\end{proof}
\begin{definition}
 Let $(V,\lambda^l,\lambda^r,\rho^l,\rho^r,\phi )$ be a representation of a Hom-Leibniz algebra $(A,\cdot, [,],\alpha).$ A linear operator $T : V\rightarrow A$ is called a relative Rota-Baxter operator on $(A,\cdot, [,],\alpha)$ with respect to $(V,\lambda^l,\lambda^r,\rho^l,\rho^r,\phi )$ if $T$ satisfies (\ref{rbHAs1}), (\ref{rbHAs2})  and (\ref{rbHL2}), i.e.,
\begin{eqnarray}
 && T\phi=\alpha T \nonumber\\
 && Tu\cdot Tv=T(\lambda^l(Tu)+\lambda^r(Tv)u) 
 \mbox{ for all $u,v\in V$}\nonumber\\ 
  && [Tu,Tv]=T(\rho^l(Tu)+\rho^r(Tv)u) \mbox{ for all $u,v\in V$} \nonumber
\end{eqnarray}
\end{definition}
Observe that as Hom-associative and Hom-Leibniz algebras case, Rota-Baxter operators on Hom-Leibniz Poisson algebras are relative Rota-Baxter operators with
respect to the regular representation.
\begin{example}
 Let $(A, \cdot, [,], \alpha)$ be a Hom-Leibniz Poisson algebra and $(V,\lambda^l,\lambda^r,\rho^l,\rho^r, \phi)$ be a representation of $(A, \cdot, [,], \alpha).$
It is easy to verify that $A\oplus V$ is a representation of $(A, \cdot, [,], \alpha)$ under the maps $\lambda_{A\oplus V}^l,\lambda_{A\oplus V}^r,\rho_{A\oplus V}^l,\rho_{A\oplus V}^r: A\rightarrow gl(A\oplus V)$ defined by
\begin{eqnarray}
 &&\lambda_{A\oplus V}^l(a)(b+v):=a\cdot v+\lambda^l(a)v;\ \lambda_{A\oplus V}^r(a)(b+v):=\lambda^r(a)v; \nonumber\\
 &&\rho_{A\oplus V}^l(a)(b+v):=[a,b]+\rho^l(a)v; \ \rho_{A\oplus V}^r(a)(b+v):=\rho^r(a)v.\nonumber
\end{eqnarray}
Define the linear map $T: A\oplus V\rightarrow A, a+v\mapsto a.$ Then $T$ is a relative Rota-Baxter operator on $A$ with
respect to the representation $A\oplus V.$
\end{example}
Also, Example \ref{exHAsr} and Example \ref{exHLr} give rise to the following example of relative Rota-Baxter on a Hom-Leibniz Poisson algebra.
\begin{example} Consider the $2$-dimensional  Hom-Leibniz Poisson algebra 
$(A,\cdot,[,],\alpha)$  where the non-zero products with respect to a basis $(e_1,e_2)$ are given by: $e_1\cdot e_2=e_2\cdot e_1:=-e_1,\ e_2\cdot e_2:=e_1+e_2;\ [e_1,e_2]=-[e_2,e_1]:=e_1$ and $\alpha(e_1):=-e_1,\ \alpha(e_2):=e_1+e_2.$
\\
 Then, one can prove that the zero-map  is the only  relative Rota-Baxter operator on $(A,\cdot, [,]\alpha)$  with respect to the regular representation. 
\end{example}
As Hom-associative algebras case \cite{tcsmam}, let give some characterizations of relative Rota-Baxter operators on Hom-Leibniz Poisson algebras.
\begin{proposition}
 A linear map $T : V\rightarrow A$ is a relative Rota-Baxter operator on a Hom-Leibniz Poisson algebra $(A,\cdot,[,],\alpha)$
with respect to the representation $(V,\lambda^l,\lambda^r,\rho^l,\rho^r,\phi)$ if and only if the graph of $T,$
$$G_r(T):=\{(T(v), v), v\in  V\}$$
is a subalgebra of the semi-direct product algebra $A\oplus V.$
\end{proposition}
The following result shows that a relative Rota-Baxter operator can be lifted up the
Rota-Baxter operator.
\begin{proposition}
 Let $(A, \cdot, [,], \alpha)$ be a Hom-Leibniz Poisson algebra, $(V,\lambda^l,\lambda^r,\rho^l,\rho^r,\phi)$ be a representation of $A$ and
$T : V\rightarrow A$ be a linear map. Define 
$\widehat{T}\in End(A\oplus V)$ by 
$\widehat{T}(a+v):=Tv.$ Then T is a relative Rota-Baxter operator
if and only $\widehat{T}$ is a Rota-Baxter operator on $A\oplus V.$
\end{proposition}
In order to give another characterization of relative Rota-Baxter operators, let introduce the following:
\begin{definition}
 Let $(A,\cdot, [,],\alpha)$ be a Hom-Leibniz Poisson algebra. A linear map 
 $N : A\rightarrow A$ is
said to be a Nijenhuis operator if $N\alpha=\alpha N$ and its Nijenhuis torsions vanish, i.e.,
\begin{eqnarray}
 && N(x)\cdot N(y)=N(N(x)\cdot y + x\cdot N(y)-N(x\cdot y)), \mbox{ for all $x, y\in A,$}\nonumber\\
 && [N(x),N(y)]=N([N(x),y] + [x,N(y)]-N([x,y])), \mbox{ for all $x, y\in A.$}\nonumber
\end{eqnarray}
Observe that the deformed multiplications
$\cdot_N, [,]_N: A\oplus A\rightarrow A$ given by
\begin{eqnarray}
 && x\cdot_N y:= N(x)\cdot y+x\cdot N(y)-N(x\cdot y),\nonumber\\
 && [x,y]_N := [N(x),y]+[x,N(y)]-N([x,y])\nonumber
\end{eqnarray}
gives rise to a new Hom-Leibniz Poisson multiplications on $A,$ and $N$ becomes a 
morphism from the Hom-Leibniz Poisson algebra $(A,\cdot_N, [,]_N,\alpha)$ to the initial Hom-Leibniz Poisson algebra $(A,\cdot, [,],\alpha).$
\end{definition}
Now, we can esealy check the following result.
\begin{proposition}
 Let $\mathcal{A}:=(A, \cdot, [,], \alpha)$ be a Hom-Leibniz Poisson algebra and $\mathcal{V}:=(V,\lambda^l,\lambda^r,\rho^l,\rho^r, \phi)$ be a representation of $(A, \cdot, [,], \alpha).$
 A linear map $T: V\rightarrow A$ is a
  relative Rota-Baxter operator on $\mathcal{A}$ with respect to the
$\mathcal{V}$ if and only if $N_T:=\left(
\begin{array}{cc}
 0& T\\
 0& 0
\end{array}
\right)
: A\oplus V \rightarrow A\oplus V$ is a Nijenhuis operator on
the Hom-Leibniz Poisson algebra $A\oplus V.$
\end{proposition}
In the sequel, let give  some results about relative Rota-Baxter operators basing on the previous sections.
\begin{lemma}\label{lHLPRB}
 Let $T$ be a relative Rota-Baxter operator on a Hom-Leibniz Poisson algebra $(A,\cdot ,[,],\alpha)$ with respect to a representation  $(V,\lambda^l,\lambda^r,\rho^l,\rho^r,\phi ).$ If define a map $\diamond$ and a braket $[,]_T$ on $V$ by (\ref{opRBHas}) and (\ref{opRB}), i.e., 
 \begin{eqnarray}
 &&u\diamond v:=\lambda^l(Tu)v+\lambda^r(Tv)u \mbox{ for all $(u, v)\in V^{\times 2}$}\nonumber\\
  &&[u,v]_T:=\rho^l(Tu)v+\rho^r(Tv)u \mbox{ for all $(u, v)\in V^{\times 2}$} \label{opRB}\nonumber
 \end{eqnarray}
then, $(V,\diamond, [,]_T, \phi)$ is a Hom-Leibniz Poisson algebra.
\end{lemma}
\begin{proof}
We know that $(V,\diamond, \phi)$ is a Hom-associative algebra
  and $(V,[,]_T, \phi)$ is a Hom-Leibniz algebra by Lemma \ref{lHAsRB} and Lemma \ref{lHLRB} respectively. The result is obtained if we prove (\ref{cHLP}). For all $u,v,w\in V,$ using Corollary \ref{cmorpHas}  and Corollary \ref{cmorp} , we get:
  \begin{eqnarray}
   &&[u\diamond v,\phi(w)]_T=\rho^l(Tu\cdot Tv)\phi(w)+\rho^r(Tw)\lambda^l(Tu)v+\rho^r(T\phi(w))\lambda^r(Tv)u,\nonumber\\
&&\phi(u)\diamond [v,w]_T=\lambda^l(T\phi(u))\rho^l(Tv)w+\lambda^l(T\phi(u))\rho^l(Tw)v+\lambda^r([Tv,Tw])\phi(u),\nonumber\\
&&[u,w]\diamond\phi(v)=\lambda^l([Tu,Tv])\phi(v)+\lambda^r(T\phi(v))\rho^l(Tu)w+
\lambda^r(T\phi(v))\rho^r(Tw)u.\nonumber
  \end{eqnarray}
Hence, we get (\ref{cHLP}) by (\ref{rbHAs1}), (\ref{rHLP1}), (\ref{rHLP2}), (\ref{rHLP3}).
\end{proof}
\begin{corollary}\label{cmorpHLP}
 Let $T$ be a relative Rota-Baxter operator on a Hom-Leibniz Poisson algebra $(A,\cdot,[,],\alpha)$ with respect to a representation  $(V,\lambda^l,\lambda^r,\rho^l,\rho^r,\phi ).$ Then, $T$ is a morphism from the Hom-Leibniz Poisson algebra 
 $(V,\diamond, [,]_T,\phi)$ to the initial Hom-Leibniz Poisson algebra $(A,[,],\alpha).$
\end{corollary}
\begin{proof}
 It follows by Corollary \ref{cmorpHas} and 
 Corollary \ref{cmorp}
\end{proof}
\begin{theorem}\label{tHLPRB}
 Let $T$ be a relative Rota-Baxter operator on a Hom-Leibniz Poisson algebra $(A,\cdot,[,],\alpha)$ with respect to a representation  $(V,\lambda^l,\lambda^r,\rho^l,\rho^r,\phi ).$  
 Then, $(A, \overline{\lambda^l}, \overline{\lambda^r},\overline{\rho^l}, \overline{\rho^r},\alpha)$ is a representation of the Hom-Leibniz Poisson algebra
 $(V,\diamond, [,]_T,\phi)$ where $\overline{\lambda^l}, \overline{\lambda^r},\overline{\rho^l}, \overline{\rho^r}$ are defined as (\ref{rrba1}),(\ref{rrba2}), (\ref{rrb1}),(\ref{rrb2}) respectively, i.e.,
 \begin{eqnarray}
 && \overline{\lambda^l}(u)x:=Tu\cdot x-T\lambda^r(x)u   \mbox{ for all $(x,u)\in A\times V$}\nonumber\\
 && \overline{\lambda^r}(u)x:=x\cdot Tu-T\lambda^l(x)u  \mbox{ for all $(x,u)\in A\times V$}\nonumber\\
 && \overline{\rho^l}(u)x:=[Tu,x]-T\rho^r(x)u   \mbox{ for all $(x,u)\in A\times V$}\nonumber\\
 && \overline{\rho^r}(u)x:=[x,Tu]-T\rho^l(x)u  \mbox{ for all $(x,u)\in A\times V.$}\nonumber
 \end{eqnarray}
 \end{theorem}
 \begin{proof}
  We have proved that $(A, \overline{\lambda^l}, \overline{\lambda^r},\alpha)$ is a representation of the Hom-associative algebra $(V,\diamond, \phi)$  and $(A, \overline{\rho^l}, \overline{\rho^r},\alpha)$ is a representation of the Hom-Leibniz algebra $(V, [,]_T,\phi)$ by Theorem \ref{tHasRB} and Theorem \ref{tHLRB} respectively. It remains to prove $(\ref{rHLP1})-(\ref{rHLP3}).$ Let $u,v\in V$ and $x\in A.$ Then
  \begin{eqnarray}
&&\overline{\rho^r}(\phi(v))\overline{\lambda^l}(u)x= [Tu\cdot x,T\phi(v)]-[T\lambda^r(x)u,T\phi(v)]-
   T\rho^l(Tu\cdot x)\phi(v)+T\rho^l(T\lambda^r(x)u)\phi(v)\nonumber\\
   &&=[Tu\cdot x,T\phi(v)]-T\rho^l(T\lambda^r(x)u)\phi(v)-T\rho^r(T\phi(v))\lambda^r(x)u-
   T\lambda^l(T\phi(u))\rho^l(x)v\nonumber\\
   &&-T\lambda^r(\alpha(x))\rho^l(Tu)v+T\rho^l(T\lambda^r(x)u)\phi(v)
   \mbox{ ( by (\ref{rbHL1}), (\ref{rbHL2}) and (\ref{rHLP3})  )}\nonumber\\
   &&=[Tu\cdot x,T\phi(v)]-T\rho^r(T\phi(v))\lambda^r(x)u-
   T\lambda^l(T\phi(u))\rho^l(x)v-T\lambda^r(\alpha(x))\rho^l(Tu)v\nonumber.
  \end{eqnarray}
Similarly, we compute
\begin{eqnarray}
&&\overline{\lambda^l}(\phi(u))\overline{\rho^r}(v)x= T\phi(u)\cdot[x,Tv]-T\phi(u)\cdot T\rho^l(x)v-T\lambda^r([x,Tv])\phi(u)+T\lambda^r(T\rho^l(x)v)\phi(u)\nonumber\\
&&=T\phi(u)\cdot[x,Tv]-T\lambda^l(T\phi(u))\rho^l(x)v-T\lambda^r(T\rho^l(x)v)\phi(u)-T\rho^r(T\phi(v))\lambda^r(x)u\nonumber\\
&&+T\lambda^r(\alpha(x))\rho^r(Tv)u+T\lambda^r(T\rho^l(x)v)\phi(u)
   \mbox{ ( by (\ref{rbHAs1}), (\ref{rbHAs2}) and (\ref{rHLP2})  )}\nonumber\\
   &&=T\phi(u)\cdot[x,Tv]-T\lambda^l(T\phi(u))\rho^l(x)v-T\rho^r(T\phi(v))\lambda^r(x)u+T\lambda^r(\alpha(x))\rho^r(Tv)u,\nonumber
  \end{eqnarray}
\begin{eqnarray}
 &&\overline{\lambda^l}([u,v]_T)\alpha(x)=
 [Tu,Tv]\cdot\alpha(x)-T\lambda^r(\alpha(x))\rho^l(Tu)v-T\lambda^r(\alpha(x))\rho^r(Tv)u.\nonumber
\end{eqnarray}
It follows that
\begin{eqnarray}
 &&\overline{\rho^r}(\phi(v))\overline{\lambda^l}(u)x-\overline{\lambda^l}(\phi(u))\overline{\rho^r}(v)x-\overline{\lambda^l}([u,v]_T)\alpha(x)\nonumber\\
 &&=\Big([Tu\cdot x,T\phi(v)]-[Tu,Tv]\cdot\alpha(x)- T\phi(u)\cdot[x,Tv]\Big)=0
 \mbox{  (  by (\ref{rbHAs1}) and (\ref{cHLP})  )} \nonumber
\end{eqnarray}
i.e., (\ref{rHLP1})  holds. To get (\ref{rHLP2}), we proceed as follows:
\begin{eqnarray}
&&\overline{\rho^r}(\phi(v))\overline{\lambda^r}(u)x= [x\cdot Tu\cdot ,T\phi(v)]-[T\lambda^l(x)u,T\phi(v)]-
   T\rho^l(x\cdot Tu)\phi(v)+T\rho^l(T\lambda^l(x)u)\phi(v)\nonumber\\
   &&=[x\cdot Tu,T\phi(v)]-T\rho^l(T\lambda^l(x)u)\phi(v)-T\rho^r(T\phi(v))\lambda^l(x)u-
   T\lambda^l(\alpha(x))\rho^l(Tu)v\nonumber\\
   &&-T\lambda^r(T\phi(u))\rho^l(x)v+T\rho^l(T\lambda^l(x)u)\phi(v)
   \mbox{ ( by (\ref{rbHL1}), (\ref{rbHL2}) and (\ref{rHLP3})  )}\nonumber\\
   &&=[x\cdot Tu,T\phi(v)]-T\rho^r(T\phi(v))\lambda^l(x)u-
   T\lambda^l(\alpha(x))\rho^l(Tu)v-T\lambda^r(T\phi(u))\rho^l(x)v\nonumber.
  \end{eqnarray}
Similarly, we compute
\begin{eqnarray}
 &&\overline{\lambda^r}([u,v]_T)\alpha(x)=
 \alpha(x)\cdot[Tu,Tv]-T\lambda^l(\alpha(x))\rho^l(Tu)v-T\lambda^l(\alpha(x))\rho^r(Tv)u,\nonumber
\end{eqnarray}
\begin{eqnarray}
&&\overline{\lambda^r}(\phi(u))\overline{\rho^r}(v)x= [x,Tv]\cdot T\phi(u)-(T\rho^l(x)v)\cdot T\phi(u)-T\lambda^l([x,Tv])\phi(u)+T\lambda^l(T\rho^l(x)v)\phi(u)\nonumber\\
&&=[x,Tv]\cdot T\phi(u)-T\lambda^l(T\rho^l(x)v)\phi(u)-T\lambda^r(T\phi(u))\rho^l(x)v-T\rho^r(T\phi(v))\lambda^l(x)u\nonumber\\
&&+T\lambda^l(\alpha(x))\rho^r(Tv)u+T\lambda^l(T\rho^l(x)v)\phi(u)
   \mbox{ ( by (\ref{rbHAs1}), (\ref{rbHAs2}) and (\ref{rHLP1})  )}\nonumber\\
   &&=[x,Tv]\cdot T\phi(u)-T\lambda^r(T\phi(u))\rho^l(x)v-T\rho^r(T\phi(v))\lambda^l(x)u+T\lambda^l(\alpha(x))\rho^r(Tv)u. \nonumber
  \end{eqnarray}
It follows that
\begin{eqnarray}
 &&\overline{\rho^r}(\phi(v))\overline{\lambda^r}(u)x-\overline{\lambda^r}([u,v]_T)\alpha(x)-\overline{\lambda^r}(\phi(u))\overline{\rho^r}(v)x\nonumber\\
 &&=\Big([x\cdot Tu,T\phi(v)]-\alpha(x)\cdot[Tu,Tv]- [x,Tv]\cdot T\phi(u)\Big)=0
 \mbox{  (  by (\ref{rbHAs1}) and (\ref{cHLP})  )} \nonumber
\end{eqnarray}
i.e., (\ref{rHLP2})  holds. Finally, we compute as the previous case:
\begin{eqnarray}
 &&\overline{\rho^l}(u\diamond v)\alpha(x)
 =[Tu\cdot Tv,\alpha(x)]-T\rho^r(\alpha(x))\lambda^l(Tu)v-T\rho^r(\alpha(x))\lambda^r(Tv)u\nonumber\\
 &&=[Tu\cdot Tv,\alpha(x)]-T\lambda^l(T\phi(u))\rho^r(x)v-T\lambda^l([Tu,x])v-T\lambda^r([Tv,x])\phi(u)\nonumber\\
 &&-T\lambda^r(T\phi(v))\rho^r(x)u \mbox{ 
 (  by (\ref{rbHL1}) (\ref{rHLP1}) and (\ref{rHLP2})).}
\end{eqnarray}
Also, we compute
\begin{eqnarray}
 &&\overline{\lambda^l}(\phi(u))\overline{\rho^l}(v)x=T\phi(u)\cdot[Tv,x]-
 T\phi(u)\cdot T\rho^r(x)v-T\lambda^r([Tv,x])\phi(u)+T\lambda^r(T\rho^r(x)v)\phi(u)\nonumber\\
 &&=T\phi(u)\cdot[Tv,x]-
 T\lambda^l(T\phi(u))\rho^r(x)v-T\lambda^r(T\rho^r(x)v)\phi(u)-T\lambda^r([Tv,x])\phi(u)\nonumber\\
 &&+T\lambda^r(T\rho^l(x)v)\phi(u) \mbox{ ( by (\ref{rbHAs2})}\nonumber\\
 &&=T\phi(u)\cdot[Tv,x]-
 T\lambda^l(T\phi(u))\rho^r(x)v-T\lambda^r([Tv,x])\phi(u)\nonumber
\end{eqnarray}
and similarly using the same hypothesis as in the previous equation, we obtain
\begin{eqnarray}
 &&\overline{\lambda^r}(\phi(v))\overline{\rho^l}(u)x=[Tu,x]\cdot T\phi(v)-T\lambda^r(T\phi(v))\rho^r(x)u-T\lambda^l([Tu,x])\phi(v).\nonumber
\end{eqnarray}
Hence, (\ref{rHLP3}) follows by (\ref{rbHL1}) and (\ref{cHLP}).
 \end{proof} 

\vspace*{1cm}
Sylvain Attan\\
 D\'epartement de Math\'ematiques, Universit\'{e} d'Abomey-Calavi
01 BP 4521, Cotonou 01, B\'enin. E.mail: syltane2010@yahoo.fr

\begin{thebibliography}{99}
\bibitem{fazeam} F. Ammar, Z. Ejbehi and A. Makhlouf, {\it Cohomology and deformations of Hom-algebras,} J. Lie Theory {\bf 21} (2011), no. 4, 813-836.
\bibitem{sahhbk} S. Attan, H. Hounnon, B. Kpamègan, {\it Representations and formal deformations of Hom-Leibniz algebras}, Asian-European Journal of Mathematics
(2020) 2150013 (14 pages).
\bibitem{gb} G. Baxter, {\it An analytic problem whose solution follows from a simple algebraic identity,} Pacific J. Math. {\bf 10} (1960) 731-742.
\bibitem{mctd} M. Casas and T. Datuashvili,     
 {\it Noncommutative Leibniz-Poisson Algebras,} Communications in Algebra, 
 {\bf 34}(2006), n0. 7, 2507-2530.
 \bibitem{ycys} Y. Cheng and Y. Su, {\it (Co)homology and universal central extension of Hom-Leibniz
algebras,} Acta Math. Sin. (Engl. Ser.) {\bf 27} (2011), 813-830.
\bibitem{tcsmam} T. Chtioui, S. Mabrouk, A. Makhlouf, {\it Cohomology and deformations of $\mathcal{O}$-operators on
Hom-associative algebras}, arXiv: 2104.10724 v1.
\bibitem{acdk} A. Connes and D. Kreimer, {\it Renormalization in quantum field theory and the Riemann-Hilbert problem. I. The Hopf algebra structure of graphs and the main theorem,} Comm. Math. Phys. {\bf 210} (2000) 249-273.
\bibitem{sgsw} S. Guo, S. Wang, {\it 
Hom-Leibniz bialgebras and relative Rota-Baxter operators}
%\bibitem{scig} S. Caenepeel and I. %Goyvaerts, Monoidal Hom-Hopf algebras, %{\it Comm. Algebra.} {\bf 39} (2011), %2216-2240.
\bibitem{HAR1} J. T. Hartwig, D. Larsson and S. D. Silvestrov, {\it Deformations of Lie algebras using $\sigma$-derivations,} J. Algebra, \textbf{292} (2006), 314-361.
\bibitem{mnhhs} M. N. Hounkonnou, G. D. Houndedji, S. Silvestrov, {\it Double constructions of biHom-Frobenius
algebras}, arXiv: 2008.06645v1.
\bibitem{dlsds1} D. Larsson, S.D. Silvestrov, {\it Quasi-Hom-Lie algebras, central extensions and 2-cocycle-like identities,} J. Algebra {\bf 288} (2005), 321-344.
\bibitem{dlsds2} D. Larsson, S.D. Silvestrov, {\it Quasi-Lie algebras,} in "Noncommutative Geometry and Representation Theory in Mathematical Physics", Contemp.
Math., 391, Amer. Math. Soc., Providence, RI, 2005, 241-248.
\bibitem{jlLlP} J. L. Loday, T. Pirashvili, {\it Universal enveloping algebras of Leibniz algebras and (co)homology.} Math.
Ann., {\bf 296}(1993), 139-158.
\bibitem{amss} A. Makhlouf and S. Silvestrov, {\it Notes on 1-parameter formal deformations
of Hom-associative and Hom-Lie algebras}, Forum Math. {\bf 22} (2010), no. 4,
715-739.

\bibitem{MAK3} A. Makhlouf, S. D. Silvestrov, {\it Hom-algebra structures,} J. Gen. Lie Theory Appl. {\bf 2} (2008), 51-64.
\bibitem{gmrs} G. Mukherjee and R. Saha, {\it Equivariant one-parameter formal deformations of Hom-Leibniz algebras,} Accepted by Commun. Contemp. Math. (2020).
\bibitem{ysrt} Y. Sheng and R. Tang, {\it Leibniz bialgebras, relative Rota-Baxter operators and the classical Leibniz Yang-Baxter equation,} preprint (2019),
arXiv: 1902.03033.
\bibitem{dy1} D. Yau, {\it Non-commutative Hom-Poisson algebras,} e-Print arXiv:1010.3408 (2010).
\end{thebibliography}
\end{document}